\documentclass[aps,pre,twocolumn,groupedaddress]{revtex4-1}

\usepackage{euscript,amsmath,amssymb,amsfonts,graphicx,bm,amsthm}

  \usepackage{paralist}
  \usepackage{graphics} 
 \usepackage[latin1]{inputenc}
\usepackage[caption=false]{subfig}
\usepackage{soul}

\usepackage{latexsym,amssymb,enumerate,amsmath,amsthm} 
  \usepackage{graphicx} 
  \usepackage{tikz}
   \usepackage[colorlinks=false]{hyperref}
\usepackage{latexsym,amssymb,enumerate,amsmath} 
\usepackage{pgfplots}
\usepackage{soul,xcolor}

\newcommand{\s}{\textbf{S}}
\newcommand{\Sb}{S_{\text{b}}}
\newcommand{\dd}{\text{d}}
\newcommand{\deuc}{L_{\textup{euc}}}
\newcommand{\drie}{L_{\textup{rie}}}
\def\eps{\varepsilon}
\def\E{\mathbb{E}}
\def\P{\mathbb{P}}
\def\R{\mathbb{R}}
\def\L{\mathcal{L}}

\newtheorem{theorem}{Theorem}

\newtheorem{prop}[theorem]{Proposition}

\theoremstyle{plain}
\theoremstyle{remark}

\begin{document}


\title[]{Universal formula for extreme first passage statistics of diffusion}


\author{Sean D. Lawley}
\email[]{lawley@math.utah.edu}
\affiliation{Department of Mathematics, University of Utah, Salt Lake City, UT 84112 USA}


\date{\today}

\begin{abstract}
The timescales of many physical, chemical, and biological processes are determined by first passage times (FPTs) of diffusion. The overwhelming majority of FPT research studies the time it takes a single diffusive searcher to find a target. However, the more relevant quantity in many systems is the time it takes the fastest searcher to find a target from a large group of searchers. This fastest FPT depends on extremely rare events and has a drastically faster timescale than the FPT of a given single searcher. In this work, we prove a simple explicit formula for every moment of the fastest FPT. The formula is remarkably universal, as it holds for $d$-dimensional diffusion processes (i) with general space-dependent diffusivities and force fields, (ii) on Riemannian manifolds, (iii) in the presence of reflecting obstacles, and (iv) with partially absorbing targets. Our results rigorously confirm, generalize, correct, and unify various conjectures and heuristics about the fastest FPT.
\end{abstract}


\maketitle

\section{Introduction} 
Many events in physical, chemical, and biological systems are initiated when a diffusive searcher finds a target \cite{redner2001}. Investigations of such first passage times (FPTs) began with Helmholtz and Lord Rayleigh in the context of acoustics \cite{helmholtz1860, rayleigh1945} and continue with current research driven largely by biological and chemical physics \cite{benichou2008, Reingruber2009, benichou2010, holcman2014time, holcman2014, calandre2014, vaccario2015, grebenkov2016, newby2016, lindsay2017}. The overwhelming majority of these studies seek to answer the question: How long does it take a given single diffusive searcher to find a target?

However, several recent studies, reviews, and commentaries have declared a major paradigm shift in the study and application of FPTs 
\cite{basnayake2019, schuss2019, coombs2019, redner2019, sokolov2019, rusakov2019, martyushev2019, tamm2019, basnayake2019c, basnayake_extreme_2018, reynaud2015, basnayake2019fast, guerrier2018}. This work has shown that the relevant question in many systems is actually: Out of a large group of diffusive searchers, how long does it take the fastest searcher to find a target?

This paradigm shift has generated new questions, calls for further analysis, and interesting conjectures to explain the apparent redundancy in many systems \cite{schuss2019, coombs2019, redner2019, sokolov2019, rusakov2019, martyushev2019, tamm2019, basnayake2019c}. For example, this work has been invoked to explain why roughly $10^{8}$ sperm cells search for the oocyte in human fertilization, when only one sperm cell is required \cite{meerson2015, reynaud2015, redner2019}. In fact, the recently formulated ``redundancy principle'' posits that many seemingly redundant copies of an object (molecules, proteins, cells, etc.)\ are not a waste, but rather have the specific function of accelerating search processes \cite{schuss2019}.

To illustrate, consider $N\gg1$ independent and identically distributed (iid) diffusive searchers. Let $\tau_{1},\dots,\tau_{N}$ be their $N$ iid FPTs to find some target. While most studies have calculated statistics of a single FPT, $\tau_{1}$, the more relevant quantity in many systems is the time it takes the fastest searcher to find the target,
\begin{align}\label{t1n}
T_{1,N}
:=\min\{\tau_{1},\dots,\tau_{N}\}.
\end{align}
This fastest FPT, $T_{1,N}$, is called an \emph{extreme statistic} \cite{gumbel2012}, and it has a drastically faster timescale than $\tau_{1}$.

Despite the fact that the statistics of a single FPT are well understood in many scenarios, very little is known about the fastest FPT. Indeed, rigorous results have been generally limited to effectively one-dimensional domains, with mostly conjectures and heuristics for diffusion in higher dimensions \cite{weiss1983, yuste1996, yuste2000, yuste2001, redner2014, meerson2015, ro2017, basnayake2019}. 

In this work, we prove a general theorem that determines every moment of the fastest FPT as $N\to\infty$ based on the short time distribution of a single FPT. We then combine this theorem with large deviation theory to prove a formula for the moments of the fastest FPT that holds in many diverse scenarios. In particular, the formula holds for ${{d}}$-dimensional diffusion processes (i) with general space-dependent diffusivities and force (drift) fields, (ii) on a Riemannian manifold, (iii) in the presence of reflecting obstacles, and (iv) with partially absorbing targets.

To summarize, first extend the definition in \eqref{t1n} by defining the $k$th fastest FPT for $k\ge1$,
\begin{align}\label{tkn}
T_{k,N}
:=\min\big\{\{\tau_{1},\dots,\tau_{N}\}\backslash\cup_{j=1}^{k-1}\{T_{j,N}\}\big\}.
\end{align}
For any fixed $m\ge1$ and $k\ge1$, we prove that the $m$th moment of the $k$th fastest FPT satisfies
\begin{align}\label{sum}
\E[(T_{k,N})^{m}]
\sim\left(\frac{{{L}}^{2}}{4D\ln N}\right)^{m}\quad\text{as }N\to\infty,
\end{align}
where ``$f\sim g$'' means $f/g\to1$. In \eqref{sum}, $D$ is a diffusivity and ${{L}}$ is a certain geodesic distance (given below) between the searcher starting locations and the target that (i) avoids any obstacles, (ii) includes any spatial variation or anisotropy in diffusivity, and (iii) incorporates any geometry in the case of diffusion on a curved manifold. Further, the length ${{L}}$ is unaffected by forces on the diffusive searchers or a finite absorption rate at the target. The result in \eqref{sum} rigorously confirms, generalizes, corrects, and unifies various conjectures and heuristics about the fastest FPT.


\section{Main theorem}

Let $S(t):=\P(\tau_{1}>t)$ denote the survival probability of a single FPT. The survival probability of the fastest FPT is then
\begin{align*}
\P(T_{1,N}>t)
=\P(\min\{\tau_{1},\dots,\tau_{N}\}>t)
=(S(t))^{N},
\end{align*}
assuming $\tau_{1},\dots,\tau_{N}$ are iid. Now, the mean of any nonnegative random variable $Z\ge0$ is $\int_{0}^{\infty}\P(Z>z)\,\dd z$. Therefore, the mean fastest FPT is
\begin{align}\label{id}
\E[T_{1,N}]
=\int_{0}^{\infty}(S(t))^{N}\,\dd t.
\end{align}

Since $S(t)$ is a decreasing function of time, it is clear from \eqref{id} that the large $N$ asymptotics of $\E[T_{1,N}]$ are determined by the short time behavior of $S(t)$. The following theorem determines these asymptotics in terms of the short time behavior of $S(t)$ on a logarithmic scale. Throughout this work, ``$f\sim g$'' means $f/g\to1$.

\begin{theorem}\label{theorem lg}
Let $\{\tau_{n}\}_{n=1}^{\infty}$ be a sequence of iid nonnegative random variables with survival probability $S(t):=\P(\tau_{1}>t)$. Assume that
\begin{align}\label{conditiona}
\int_{0}^{\infty}(S(t))^{N}\,\dd t<\infty\quad\text{for some $N\ge1$}, 
\end{align}
and assume that there exists a constant $C>0$ so that
\begin{align}\label{conditionb}
\lim_{t\to0+}t\ln(1-S(t))=-C<0.
\end{align}
Then for any $m\ge1$ and $k\ge1$, the $m$th moment of the $k$th fastest time $T_{k,N}$ in \eqref{tkn} satisfies
\begin{align}\label{res}
\E[(T_{k,N})^{m}]
\sim\Big(\frac{C}{\ln N}\Big)^{m}\quad\text{as }N\to\infty.
\end{align}
\end{theorem}

We now sketch the proof of Theorem~\ref{theorem lg} for the case $m=k=1$. The assumption in \eqref{conditionb} means roughly that
\begin{align*}
S(t)\approx1-e^{-C/t}\quad\text{for }t\ll1.
\end{align*}
Now, for a one-dimensional, pure diffusion process with unit diffusivity starting at the origin, let $\tau(l)$ denote the first time the process escapes the interval $(-2l,2l)$. The survival probability $S_{l}(t):=\P(\tau(l)>t)$ satisfies 
\begin{align*}
S_{l}(t)\approx1-e^{-l^{2}/t}\quad\text{for }t\ll1.
\end{align*}
Therefore, taking $l_{\pm}=\sqrt{C\pm \eps}$ for small $\eps>0$ yields
\begin{align*}
S_{l_{-}}(t)
\le S(t)
\le S_{l_{+}}(t)\quad\text{for }t\ll1.
\end{align*}
Hence, for sufficiently large $N$ we have the bounds
\begin{align*}
\int_{0}^{\infty}(S_{l_{-}}(t))^{N}\,\dd t
\le\int_{0}^{\infty}(S(t))^{N}\,\dd t
\le\int_{0}^{\infty}(S_{l_{+}}(t))^{N}\,\dd t,
\end{align*}
since the large $N$ behavior of these integrals is determined by the short time behavior of their integrands (this uses \eqref{conditiona}, which ensures that $\E[T_{1,N}]<\infty$ is finite for large $N$). Furthermore, it is known that \cite{weiss1983,lawley2019gumbel}
\begin{align*}
\int_{0}^{\infty}(S_{l_{\pm}}(t))^{N}\,\dd t
\sim\frac{C\pm\eps}{\ln N}\quad\text{as }N\to\infty.
\end{align*}
Noting that $\eps$ is arbitrary completes the argument. The full proof is in Appendix~\ref{provethm1}.

\section{Applications of Theorem~\ref{theorem lg}}

In this section, we combine Theorem~\ref{theorem lg} with large deviation theory to (i) prove that \eqref{res} is remarkably universal and (ii) identify the constant $C$ in \eqref{res}.

Let $\{X(t)\}_{t\ge0}$ be a ${{d}}$-dimensional diffusion process on a manifold $M$ and let $p({{x}},t|{x_{0}},0)$ be the probability density that $X(t)={{x}}$ given 
$X(0)={x_{0}}$. That is,
\begin{align}\label{pd}
p({{x}},t|{x_{0}},0)\,\dd {{x}}
=\P(X(t)={{x}}\,|\,X(0)={x_{0}}).
\end{align}
Let $\tau>0$ be the FPT to some target set $U_{\text{T}}\subset M$,
\begin{align}\label{tau}
\tau
:=\inf\{t>0:X(t)\in U_{\text{T}}\},
\end{align}
and let $S(t):=\P(\tau>t)$ be the survival probability. Let $\{\tau_{n}\}_{n=1}^{\infty}$ be a sequence of iid realizations of $\tau$ and let $T_{k,N}$ be the $k$th order statistic in \eqref{tkn}. We assume the target $U_{\text{T}}$ is the closure of its interior, which precludes trivial cases such as the target being a single point (which would make $\tau=+\infty$ in dimension $d\ge2$). Assume the initial distribution of $X$ is a probability measure with compact support $U_{0}\subset M$ that does not intersect the target,
\begin{align}\label{away}
U_{0}\cap U_{\text{T}}=\varnothing.
\end{align}
For example, the initial distribution could be a Dirac mass at a point $X(0)=x_{0}=U_{0}\in M$ if $x_{0}\notin U_{\text{T}}$, or it could be uniform on $U_{0}$ if the closed set $U_{0}$ satisfies \eqref{away}.


\subsection{Pure diffusion in $\R^{{d}}$\label{section pure}} To setup more complicated applications of Theorem~\ref{theorem lg}, first consider the simple case of free diffusion in $M=\R^{{d}}$ with diffusivity $D>0$. Of course, the probability density \eqref{pd} is Gaussian,
\begin{align}\label{pured}
p({{x}},t|{x_{0}},0)
=\frac{1}{(4\pi D t)^{d/2}}\exp\left(\frac{-\deuc^{2}({x_{0}},{{x}})}{4Dt}\right),
\end{align}
where $\deuc({x_{0}},{{x}}):=\|{x_{0}}-{{x}}\|$ is the standard Euclidean length. A simple manipulation of \eqref{pured} reveals the following short time behavior of the probability density,
\begin{align}\label{vpure}
\lim_{t\to0+}t\ln p({{x}},t|{x_{0}},0)
=-\frac{\deuc^{2}({x_{0}},{{x}})}{4D}.
\end{align}
This behavior of the probability density implies that (see Appendix~\ref{shorttime})
\begin{align}\label{vpureS}
\lim_{t\to0+}t\ln(1-S(t))
=-\frac{\deuc^{2}(U_{0},U_{\text{T}})}{4D}<0,
\end{align}
where $\deuc(U_{0},U_{\text{T}})$ is the shortest distance from $U_{0}$ to the target,
\begin{align}\label{dset}
\deuc(U_{0},U_{\text{T}})
:=\inf_{x_{0}\in U_{0},{{x}}\in U_{\text{T}}}\deuc({x_{0}},{{x}})>0.
\end{align}
Note that \eqref{away} ensures \eqref{dset} is strictly positive.

Therefore, Theorem~\ref{theorem lg} implies \eqref{sum} for the Euclidean length ${{L}}=\deuc(U_{0},U_{\text{T}})$ if \eqref{conditiona} is satisfied. If the dimension is ${{d}}\in\{1,2\}$, then \eqref{conditiona} is satisfied for $N=3$. If ${{d}}\ge3$, then merely taking the complement of the target, $\R^{n}\backslash U_{\text{T}}$, to be bounded ensures \eqref{conditiona} is satisfied for $N=1$.

\begin{figure}[t]
\centering
\includegraphics[width=1\linewidth]{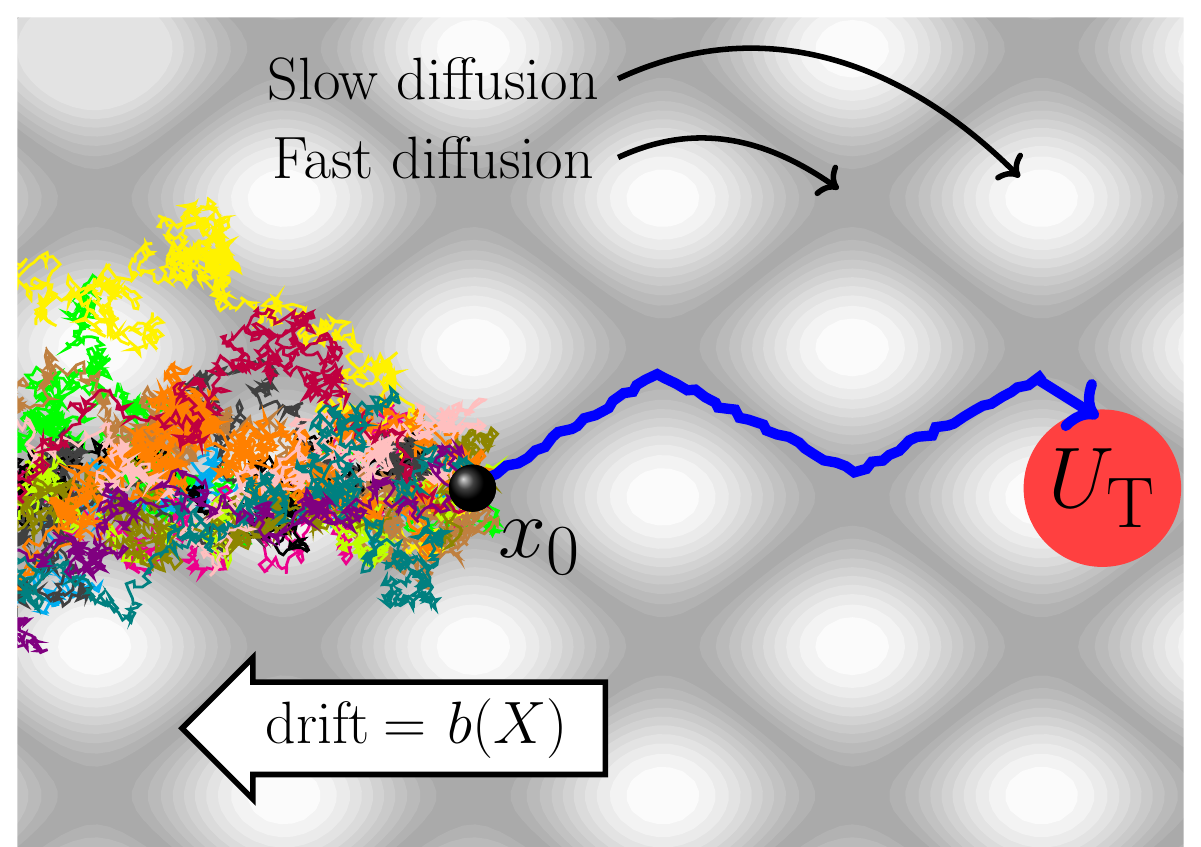}
\caption{\small Diffusion with space-dependent diffusivity and drift. For the diffusion process in \eqref{sde}, the thin trajectories are 15 typical paths which get pushed to the left by the drift. The thick blue curve shows that the fastest searcher is unaffected by the drift and moves almost deterministically to the target through regions of fast diffusion (grey regions) while avoiding regions of slow diffusion (white regions).}
\label{figgeo}
\end{figure}

\subsection{Diffusion with space-dependent diffusivity and drift}

Suppose the diffusion follows the It\^{o} stochastic differential equation on $M=\R^{{d}}$,
\begin{align}\label{sde}
\begin{split}
\dd X
&=b(X)\,\dd t+\sqrt{2D}\sigma(X)\,\dd W.
\end{split}
\end{align}
Here, $b:\R^{{d}}\to\R^{{d}}$ is the space-dependent drift vector describing any force on the searcher, $\sigma:\R^{{d}}\to\R^{{{d}}\times m}$ is a dimensionless function describing any space-dependence or anisotropy in the diffusion, $D>0$ is a characteristic diffusivity, and $W(t)\in\R^{m}$ is a standard Brownian motion. Assume $\R^{{d}}\backslash U_{\text{T}}$ is bounded to ensure \eqref{conditiona} is satisfied. Assume $b$ and $\sigma$ satisfy mild conditions (namely that $b$ is uniformly bounded and uniformly Holder continuous and that $\sigma\sigma^{T}$ is uniformly Holder continuous and its eigenvalues are bounded above $\alpha_{1}>0$ and bounded below $\alpha_{2}>\alpha_{1}$).

For any smooth parametric path $\omega:[0,1]\to M$, define the length of the path in the Riemannian metric given by the inverse of the diffusivity matrix $a:=\sigma\sigma^{T}$,
\begin{align}\label{ll}
l(\omega)
:=\int_{0}^{1}\sqrt{\dot{\omega}^{T}(s)a^{-1}(\omega(s))\dot{\omega}(s)}\,\dd s.
\end{align}
Then, the probability density \eqref{pd} satisfies \cite{varadhan1967}
\begin{align}\label{vv}
\lim_{t\to0+}t\ln p({{x}},t|{x_{0}},0)
=-\frac{\drie^{2}({x_{0}},{{x}})}{4D},
\end{align}
where $\drie$ is the geodesic length
\begin{align}\label{drie}
\begin{split}
\drie({x_{0}},{{x}})
&:=\inf\{l(\omega):\omega(0)=x_{0},\,\omega(1)=x\},
\end{split}
\end{align}
where the infimum is over smooth paths $\omega:[0,1]\to M$ which connect $\omega(0)={x_{0}}$ to $\omega(1)={{x}}$. Equation~\eqref{vv} is a celebrated result in large deviation theory known as Varadhan's formula \cite{varadhan1967,norris1997}, which generalizes the elementary formula in \eqref{vpure}. Intuitively, $\drie(x_{0},x)$ is the length of the optimal path from $x_{0}$ to $x$, where paths are penalized for passing through regions of slow diffusion, see Fig.~\ref{figgeo}. Notice that $\drie$ reduces to the Euclidean length $\deuc$ if $a$ is the identity matrix.

Varadhan's formula \eqref{vv} implies (see Appendix~\ref{shorttime})
\begin{align}\label{vvS}
\lim_{t\to0+}t\ln(1-S(t))
=-\frac{\drie^{2}({U_{0}},U_{\text{T}})}{4D}<0,
\end{align}
where $\drie({U_{0}},U_{\text{T}})$ is defined analogously to \eqref{dset},
\begin{align}\label{dset2}
\drie({U_{0}},U_{\text{T}})
:=\inf_{x_{0}\in U_{0},{{x}}\in U_{\text{T}}}\drie({x_{0}},{{x}})>0,
\end{align}
and is strictly positive by \eqref{away}. Therefore, Theorem~\ref{theorem lg} implies that the extreme FPT formula \eqref{sum} holds for the length ${{L}}=\drie(U_{0},U_{\text{T}})$.

Hence, the drift $b$ in \eqref{sde} has no effect on extreme FPTs. This counterintuitive result confirms a conjecture of Weiss, Shuler, and Lindenberg \cite{weiss1983}. Furthermore, \eqref{drie} reveals how extreme FPTs depend on heterogeneous diffusion. In particular, \eqref{drie} shows that the fastest searchers avoid regions of space in which the diffusivity is slow. These two points are illustrated in Fig.~\ref{figgeo}.

\begin{figure}[t]
\centering
\includegraphics[width=1\linewidth]{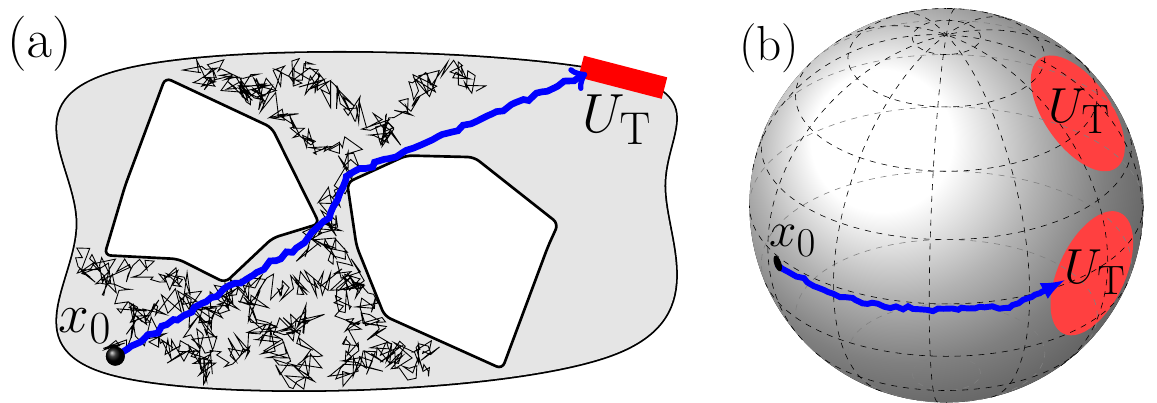}
\caption{\small Diffusion on a manifold with reflecting obstacles. (a) The thin black trajectory shows a typical diffusive path that wanders around before finding the target. The thick blue trajectory illustrates that the fastest searcher moves almost deterministically along the shortest path to the target while avoiding any obstacles. (b) The blue trajectory shows that the fastest searcher follows the shortest path to the target, which depends on the curvature of the manifold. If the target is multiple regions, the fastest searcher finds the closest one.}
\label{figobs}
\end{figure}

\subsection{Diffusion on a manifold with reflecting obstacles}

Let $M$ be a ${{d}}$-dimensional smooth Riemannian manifold. As two simple examples, $M$ could be a set in $\R^{{d}}$ with smooth outer and inner boundaries (obstacles) as in Fig.~\ref{figobs}a, or $M$ could be the surface of a 3-dimensional sphere as in Fig.~\ref{figobs}b. Consider a diffusion process on $M$ described by its generator $\L$, which in each coordinate chart is a second order differential operator of the form
\begin{align*}
\L f
=D\sum_{i,j=1}^{n}\frac{\partial}{\partial x_{i}}\Big(a_{ij}(x)\frac{\partial f}{\partial x_{j}}\Big),
\end{align*}
where $a=\{a_{ij}\}_{i,j=1}^{n}$ satisfies some mild conditions (assume that in each chart, $a$ is symmetric, continuous, and its eigenvalues are bounded above some $\alpha_{1}>0$ and bounded below some $\alpha_{2}>\alpha_{1}$). Assume the diffusion reflects from the boundary of $M$ (if $M$ has a boundary) and assume $M$ is connected and compact to ensure \eqref{conditiona} is satisfied.

In this setup, the probability density \eqref{pd} satisfies \cite{norris1997}
\begin{align}\label{vvnorris}
\lim_{t\to0+}t\ln p({{x}},t|{x_{0}},0)
=-\frac{\drie^{2}({x_{0}},{{x}})}{4D},
\end{align}
where the length is again given by \eqref{drie}, and thus (see Appendix~\ref{shorttime})
\begin{align}\label{vvnorrisS}
\lim_{t\to0+}t\ln(1-S(t))
=-\frac{\drie^{2}({U_{0}},U_{\text{T}})}{4D}<0,
\end{align}
which is strictly negative by \eqref{away}. Therefore, Theorem~\ref{theorem lg} implies that the extreme FPTs satisfy \eqref{sum}.

Fig.~\ref{figobs}a illustrates that the fastest searcher takes the shortest path to the target while avoiding any obstacles. Note that the infimum in \eqref{drie} is over smooth paths which lie in $M$, and thus paths which go through obstacles are excluded. Fig.~\ref{figobs}b illustrates that the fastest searcher takes the shortest path to the target, where the length depends on the curvature of the manifold. Fig.~\ref{figobs}b also illustrates that if the target $U_{\text{T}}$ consists of multiple regions, the fastest searcher finds the closest target.

\subsection{Partially absorbing targets}

Our analysis above, and all previous work on extreme FPTs, assumes that the target is perfectly absorbing. That is, it assumes that the searcher is absorbed as soon as it hits the target. However, a more general model assumes that the target is \emph{partially absorbing}. This means that when a searcher hits the target, it is either absorbed or reflected, and the probabilities of these events are described by a parameter $\kappa>0$ called the reactivity or absorption rate \cite{grebenkov2006}.

Consider a one-dimensional pure diffusion on the positive real line with a partially absorbing target at the origin with reactivity $\kappa$. Let $\tau$ be the first time the diffusion hits the target and $\tau_{\kappa}\ge\tau$ be when it is absorbed. If $X(0)=L>0$, then an exact calculation yields
\begin{align}\label{Skappa}
S_{\kappa}(t)
:=\P(\tau_{\kappa}>t)
=S(t)+e^{\frac{\kappa  (\kappa  t+L)}{D}} \text{erfc}\Big(\frac{2 \kappa  t+L}{\sqrt{4D t}}\Big),
\end{align}
where $S(t)=\P(\tau>t)=1-\text{erfc}(\frac{L}{\sqrt{4D t}})$. Using this formula, a straightforward calculation shows that
\begin{align*}
\lim_{t\to0+}t\ln(1-S_{\kappa}(t))
=\lim_{t\to0+}t\ln(1-S(t))
=-\frac{L^{2}}{4D}.
\end{align*}

Therefore, upon noting that \eqref{conditiona} is satisfied for $N=3$, we conclude that the extreme statistics satisfy \eqref{sum}. That is, if $T_{k,N,\kappa}$ is the $k$th fastest absorption time and $T_{k,N}$ is the $k$th fastest hitting time, then as $N\to\infty$,
\begin{align}\label{kappaconv}
\E[(T_{k,N,\kappa})^{m}]\sim
\E[(T_{k,N})^{m}]
\sim\left(\frac{{{L}}^{2}}{4D\ln N}\right)^{m}.
\end{align}
Hence, the extreme statistics for a partially absorbing target and a perfectly absorbing target are identical. 

\begin{figure}[t]
\centering
\includegraphics[width=.9\linewidth]{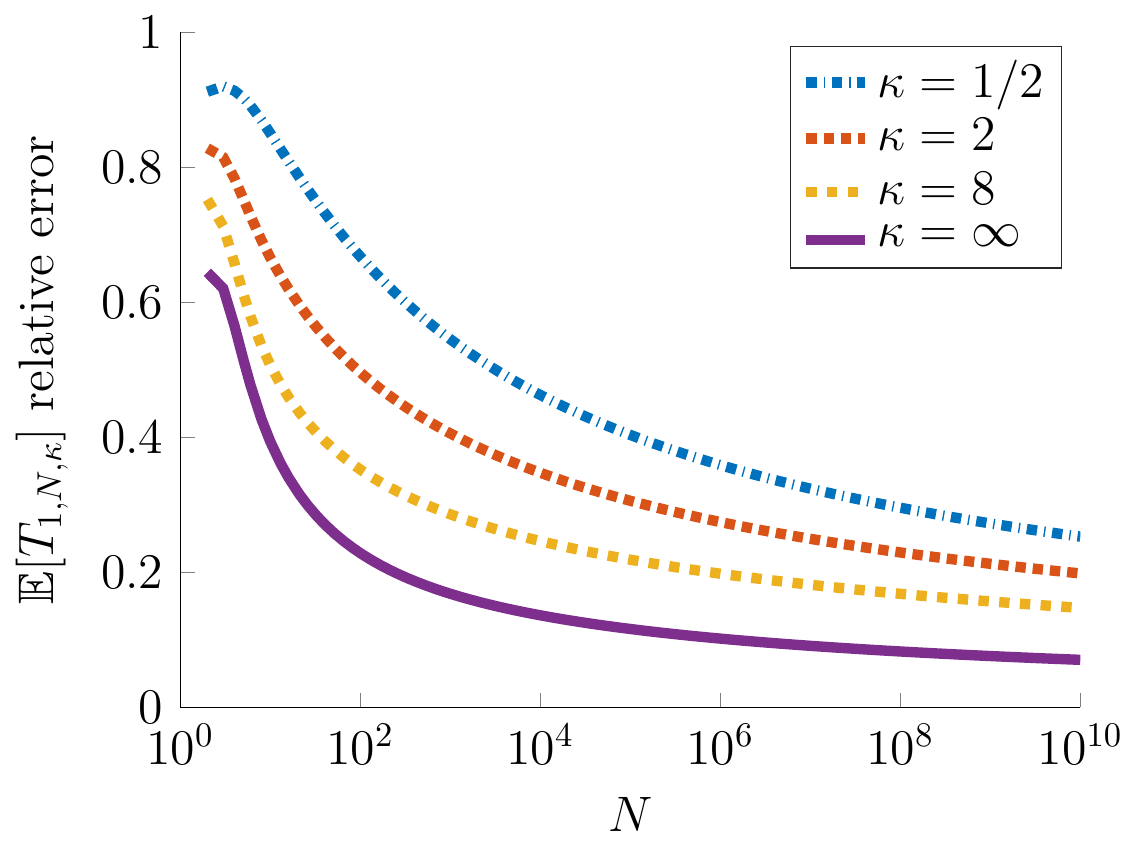}
\caption{\small Relative error \eqref{re} as a function of $N$ for different values of the reactivity $\kappa$. We take $L=D=1$.}
\label{figsim}
\end{figure}

Since \eqref{kappaconv} is statement about the large $N$ behavior of the extreme statistics, it is natural to ask about the convergence rate. Further, it is clear that for any finite $N$, and any moment $m>0$,
\begin{align*}
\E[(T_{k,N})^{m}]
<\E[(T_{k,N,\kappa_{+}})^{m}]
<\E[(T_{k,N,\kappa_{-}})^{m}],
\end{align*}
where $0<\kappa_{-}<\kappa_{+}<\infty$. It is thus also natural to ask how the convergence depends on the reactivity $\kappa$. Using the exact formula in \eqref{Skappa}, we can numerically evaluate the integral,
\begin{align}\label{quadrature}
\E[T_{1,N,\kappa}]=\int_{0}^{\infty}(S_{\kappa}(t))^{N}\,\dd t,
\end{align}
to yield a numerical approximation to $\E[T_{1,N,\kappa}]$. In Fig.~\ref{figsim}, we plot the relative error,
\begin{align}\label{re}
\left|\frac{\int_{0}^{\infty}(S_{\kappa}(t))^{N}\,\dd t-\frac{{{L}}^{2}}{4D\ln N}}{\int_{0}^{\infty}(S_{\kappa}(t))^{N}\,\dd t}\right|,
\end{align}
between our asymptotic formula in \eqref{kappaconv} and the quadrature in \eqref{quadrature} as a function of $N$. This figure illustrates that the convergence rate is slow, since the relative error is on the order of tens of percentage points for $N$ as large as $10^{10}$. Indeed, the slow convergence rate for formulas for extreme statistics of FPTs of diffusion is well known \cite{weiss1983}. This figure also shows that the error decreases as $\kappa$ increases ($\kappa=\infty$ corresponds to a perfectly absorbing target). This is to be expected, since for any fixed $N$, the FPT $T_{k,N,\kappa}$ diverges almost surely as $\kappa\to0+$, whereas the asymptotic formula in \eqref{kappaconv} is independent of $\kappa$. To see how the extreme FPT statistics depend on $\kappa$ at higher order for large $N$, see our recent work in Ref.~\cite{lawley2019gumbel}.


While the calculation which led to \eqref{kappaconv} was for a one-dimensional problem, this result that the leading order extreme statistics are independent of the reactivity $\kappa$ extends to much more general systems. To see why, observe that the absorption time, $\tau_{\kappa}$, is the sum of (i) the time, $\tau$, that it takes a searcher to first hit the target and (ii) the time, call it $\tau_{0}$, that it takes to be absorbed after starting on the target (this follows from the strong Markov property \cite{gardiner2009}). The fact that the extreme statistics are unaffected by a partially absorbing target ($\kappa<\infty$) versus a perfectly absorbing target ($\kappa=\infty$) is equivalent to $\tau_{0}\ll\tau$ for the fastest searchers. As we have seen, $\P(\tau>t)\approx1-\exp(-{{L}}^{2}/(4Dt))$ at short times, where $L>0$ depends on the domain. Further, the short time behavior of $\P(\tau_{0}>t)$ will not depend on the domain, since the problem is effectively one-dimensional at short times for a searcher starting on a partially absorbing target. Hence, the fact that $\tau_{0}\ll\tau$ for the fastest searchers in this one-dimensional problem implies that it also holds for more general systems. We make this argument rigorous in Appendix~\ref{partial}. Specifically, we prove that the extreme statistics for a partially absorbing target and a perfectly absorbing target are identical for pure diffusion in smooth bounded domains in $\R^{{d}}$ where the target is any finite disjoint union of hyperspheres.

\section{Discussion}
We have proven the formula in \eqref{sum} for the extreme FPT statistics of diffusive search, where $L=L(U_{0},U_{\text{T}})>0$ is given in \eqref{dset2} and is the geodesic distance between the possible initial searcher locations $U_{0}$ and the target $U_{\text{T}}$. This distance is the minimal length of a path that connects the initial searcher location to the target that (i) avoids any reflecting obstacles, (ii) incurs a cost for paths that go through regions of slow diffusion, and (iii) incorporates any curvature in the underlying space. Further, this distance is (iv) unaffected by a force (drift) field and (v) unaffected by a partially absorbing target.

The study of extreme FPTs of diffusion began in 1983 with Weiss, Shuler, and Lindenberg \cite{weiss1983}, where they derived $\E[T_{k,N}]\sim {{L}}^{2}/(4D\ln N)$ for one-dimensional domains with constant diffusivity and a certain class of force field. They conjectured that $\E[T_{k,N}]\sim  C/\ln N$ in higher-dimensions independent of the force field, but pointed out that they had ``nothing like a proof'' and that the constant $C$ may be ``quite difficult to calculate.'' Our results rigorously confirm their conjecture and determine $C$. Important analysis of extreme FPTs in effectively one-dimensional domains continued in \cite{yuste1996, yuste2000, yuste2001, van2003, redner2014, meerson2015}.

The recent interest in extreme FPTs of diffusion was sparked by the pioneering work in \cite{basnayake2019}, wherein the authors formally derived $\E[T_{1,N}]\sim {{L}}^{2}/(4D\ln N)$ for pure diffusion in 2-dimensional domains with small targets. Their work also found that $\E[T_{1,N}]$ decays like $1/\sqrt{\ln N}$ in 3-dimensional domains, which was later corrected for convex domains in \cite{lawley2019esp}. In fact, the correct 3-dimensional result for small targets was first derived in \cite{ro2017}.

The importance of extreme FPTs of diffusion in molecular and cellular biology was recently highlighted in the excellent review \cite{schuss2019}. This review prompted 7 subsequent commentaries \cite{coombs2019, redner2019, sokolov2019, rusakov2019, martyushev2019, tamm2019, basnayake2019c}, which each emphasized different aspects of how extreme statistics transform traditional notions of biological timescales. These commentaries also noted the need for further analysis of extreme FPTs.

The results in this work significantly extend the previous results on extreme FPTs. Indeed, most prior work considered only pure diffusion in either effectively one-dimensional domains or domains with small targets. In contrast, our results allow general space-dependent diffusivities and force fields with general targets, diffusion on manifolds with obstacles, and partially absorbing targets. In addition, our analysis yields every moment of the extreme FPTs, rather than only the mean. Indeed, \eqref{sum} implies that the variance vanishes faster than $(\ln N)^{-2}$,
\begin{align*}
\text{Variance}(T_{k,N})=o((\ln N)^{-2})\quad\text{as }N\to\infty.
\end{align*}

 In further contrast, prior analysis tended to rely on exact formulas for certain probabilities which are known only for simple domains or complicated formal asymptotics. The present work unites and extends this previous work with a simple and rigorous argument.

It is well known that intracellular \cite{blum1989} and extracellular \cite{sykova2008, nicholson2017} domains are very tortuous. This tortuosity is commonly modeled by heterogeneous diffusivity \cite{cherstvy2013}, reflecting obstacles \cite{blum1989}, and/or an effective force field that tends to exclude searchers from regions of dense obstacles \cite{isaacson2011}. Hence, this work has direct relevance to these models. Indeed, a number of influential works have found that tortuous and crowded geometries drastically affect FPTs of \emph{single} searchers \cite{benichou2010geometry, isaacson2011, woringer2014}. For example, Ref.~\cite{isaacson2011} used microscopic imaging of a nucleus to determine how volume exclusion by chromatin affects the time it takes a regulatory protein to find specific binding sites (the chromatin was modeled by an effective force field). Since we have proven that extreme FPTs are unaffected by force fields and depend only on the shortest path that avoids obstacles and regions of slow diffusivity, we predict that tortuous domains have a much weaker effect on processes initiated by the fastest searcher out of many searchers.

While the present work computes every moment of $T_{k,N}$ assuming merely that \eqref{conditionb} holds, one can obtain more information about the distribution of $T_{k,N}$ if we have more information about $S(t)$ at short time. Specifically, we have recently proven \cite{lawley2019gumbel} that if
\begin{align*}
1-S(t)
\sim At^{p}e^{-C/t}\quad\text{as }t\to0+,
\end{align*}
for some $A>0$, $p\in\R$, $C>0$, then a certain rescaling of $T_{k,N}$ converges in distribution to a type of Gumbel random variable. In addition to giving the distribution of $T_{k,N}$, the results in \cite{lawley2019gumbel} yield higher order terms in the moment formulas obtained in the present work.

Finally, a remarkable feature of extreme FPTs of diffusion is that the fastest searchers are almost \emph{deterministic}, as they tend to follow the shortest path to the target. This point has been argued heuristically, beginning in \cite{weiss1983} and continuing with recent work \cite{basnayake_extreme_2018}. 

The point that the fastest searchers move almost deterministically along the shortest path to the target is clear from our formula \eqref{sum} upon noting that the length ${{L}}$ in the formula is a ``local'' quantity that depends only on properties near this shortest path. That is, extreme FPTs are independent of perturbations outside any small region around this path (as long as these perturbations do not create a shorter path). Indeed, taking the diffusivity to be arbitrarily small away from this path does not affect extreme FPTs. Of course, this can only be true if the fastest searchers follow the shortest path.

While the asymptotically deterministic behavior of extreme first passage processes stems from the large number of searchers, this phenomenon is very different from the law of large numbers. In the law of large numbers, the deterministic behavior arises through averaging many random samples. In contrast, the deterministic behavior in extreme first passage theory occurs through rare events. This is a manifestation of the well known principle in large deviations that rare events occur in a predictable fashion; they are controlled by the least unlikely scenario.

\medskip
\begin{acknowledgments}
The author was supported by the National Science Foundation (Grant Nos.\ DMS-1814832 and DMS-1148230).
\end{acknowledgments}

\appendix

\section{Proof of Theorem~\ref{theorem lg}\label{provethm1}}

Before proving Theorem~\ref{theorem lg}, we first prove a slightly different result.

\begin{prop}\label{theorem vg}
Assume $S:[0,\infty)\to[0,1]$ is a nonincreasing function satisfying
\begin{enumerate}[(a)]
\item
$\int_{0}^{\infty}(S(t))^{N}\,\dd t<\infty$ for some $N\ge1$,
\item
there exists a constant $C>0$ so that
\begin{align*}
\lim_{t\to0+}t\ln(1-S(t))=-C<0.
\end{align*}
\end{enumerate}
Then for each $m\ge1$, we have that
\begin{align*}
\int_{0}^{\infty}(S(t^{1/m}))^{N}\,\dd t
\sim\Big(\frac{C}{\ln N}\Big)^{m}\quad\text{as }N\to\infty.
\end{align*}
\end{prop}

\begin{proof}[Proof of Proposition~\ref{theorem vg}]
For $l>0$, let $\tau(l)$ denote the first time a one-dimensional diffusion process with unit diffusivity starting at the origin escapes the interval $(-2l,2l)$. The survival probability $S_{l}(t)=\P(\tau(l)>t)$ satisfies
\begin{align}\label{gl}
\lim_{t\to0+}t\ln(1-S_{l}(t))=-l^{2}<0.
\end{align}
Let $\eps\in(0,C)$ and define $l_{\pm}:=\sqrt{C\pm\eps}$. By \eqref{gl} and assumption (b) of the proposition, there exists a $\delta>0$ so that
\begin{align*}
S_{l_{-}}(t^{1/m})
\le S(t^{1/m})
\le S_{l_{+}}(t^{1/m})\quad\text{for all }t\in[0,\delta].
\end{align*}
Therefore,
\begin{align}
\begin{split}\label{squeeze}
&\int_{0}^{\delta}(S_{l_{-}}(t^{1/m}))^{N}\,\dd t+\int_{\delta}^{\infty}(S(t^{1/m}))^{N}\,\dd t\\
&\quad\le \int_{0}^{\infty}(S(t^{1/m}))^{N}\,\dd t\\
&\qquad\le\int_{0}^{\delta}(S_{l_{+}}(t^{1/m}))^{N}\,\dd t+\int_{\delta}^{\infty}(S(t^{1/m}))^{N}\,\dd t.
\end{split}
\end{align}

Now, a simple change of variables shows that
\begin{align*}
\int_{0}^{\infty}(S(t^{1/m}))^{N}\,\dd t
=m\int_{0}^{\infty}t^{m-1}(S(t))^{N}\,\dd t\quad\text{if }m\ge1.
\end{align*}
By assumption (a) of the proposition, there exists an $N_{0}\ge1$ so that $\int_{0}^{\infty}(S(t))^{N_{0}}\,\dd t<\infty$. It is then straightforward to check that
\begin{align*}
m\int_{0}^{\infty}t^{m-1}(S(t))^{2^{m-1}N_{0}}\,\dd t<\infty.
\end{align*}
Hence, if $N\ge N_{0}$, we have that since $S$ is nonincreasing,
\begin{align*}
\int_{\delta}^{\infty}(S(t^{1/m}))^{N}\,\dd t
\le K_{0}(S(\delta^{1/m}))^{N},
\end{align*}
where $S(\delta^{1/m})<1$ by assumption (b) of the proposition, and
\begin{align*}
K_{0}
=\int_{\delta}^{\infty}\Big(\frac{S(t^{1/m})}{S(\delta^{1/m})}\Big)^{N_{0}}\,\dd t<\infty.
\end{align*}
Thus,
\begin{align*}
\lim_{N\to\infty}(\ln N)^{m}\int_{\delta}^{\infty}(S(t^{1/m}))^{N}\,\dd t=0.
\end{align*}

Therefore, multiplying \eqref{squeeze} by $(\ln N)^{m}$ and taking $N\to\infty$ yields
\begin{align*}
(C-\eps)^{m}
&\le \liminf_{N\to\infty}\frac{\int_{0}^{\infty}(S(t^{1/m}))^{N}\,\dd t}{(1/\ln N)^{m}}\\
&\quad\le\limsup_{N\to\infty}\frac{\int_{0}^{\infty}(S(t^{1/m}))^{N}\,\dd t}{(1/\ln N)^{m}}
\le(C+\eps)^{m},
\end{align*}
since it is known \cite{weiss1983,lawley2019gumbel} that for any $l>0$,
\begin{align*}
\int_{0}^{\delta}(S_{l}(t^{1/m}))^{N}\,\dd t
&\sim\int_{0}^{\infty}(S_{l}(t^{1/m}))^{N}\,\dd t\\
&\sim\Big(\frac{l^{2}}{\ln N}\Big)^{m}\quad\text{as }N\to\infty.
\end{align*}
Since $\eps\in(0,C)$ was arbitrary, the proof is complete.
\end{proof}

\begin{proof}[Proof of Theorem~\ref{theorem lg}]
Since the mean of any nonnegative random variable $Z\ge0$ is $\int_{0}^{\infty}\P(Z>z)\,\dd z$, we have that
\begin{align*}
\E[(T_{k,N})^{m}]=\int_{0}^{\infty}\P(T_{k,N}>t^{1/m})\,\dd t.
\end{align*}
For $k\in\{1,\dots,N\}$, it is immediate that
\begin{align*}
\P(T_{k,N}>t)
&=\P(T_{1,N}>t)+\P(T_{1,N}<t<T_{2,N})\\
&\quad+\dots+\P(T_{k-1,N}<t<T_{k,N}).
\end{align*}
Furthermore, we have that $\P(T_{1,N}>t)=(S(t))^{N}$ and
\begin{align*}
\P(T_{j,N}<t<T_{j+1,N})
={N\choose j}(1-S(t))^{j}S(t)^{N-j},
\end{align*}
for $j\in\{1,\dots,k-1\}$.

Now, it is straightforward to check that if $j\in\{1,\dots,k-1\}$, then
\begin{align*}
\lim_{N\to\infty}\frac{\int_{0}^{\infty}{N\choose j}(1-S(t^{1/m}))^{j}S(t^{1/m})^{N-j}\,\dd t}{(1/\ln N)^{m}}=0.
\end{align*}
Therefore, since Proposition~\ref{theorem vg} implies that
\begin{align*}
\lim_{N\to\infty}\frac{\int_{0}^{\infty}(S(t^{1/m}))^{N}\,\dd t}{(C/\ln N)^{m}}=1,
\end{align*}
the proof is complete.
\end{proof}

\section{Probability density and survival probability at short time\label{shorttime}}

Varadhan's formula \cite{varadhan1967,norris1997} gives the short time behavior of the probability density of a diffusive searcher in terms of a certain geodesic distance (see \eqref{vpure}, \eqref{vv}, and \eqref{vvnorris}). However, the assumptions of Theorem~\ref{theorem lg} require the short time behavior of the survival probability rather than the probability density. Here, we show how the short time behavior of the probability density yields the short time behavior of the survival probability.

\begin{proof}[Proof that \eqref{vpure} implies \eqref{vpureS}]
Consider the case of free diffusion in $M=\R^{{d}}$ with diffusivity $D>0$ (section~\ref{section pure}). We first bound $1-S(t)$ from below. Notice that if the process is at the target at time $t$, then certainly $\tau\le t$. That is, 
\begin{align}\label{bfb}
\begin{split}
1-S(t)
=\P(\tau<t)
&\ge\P(X(t)\in U_{\text{T}})\\
&=\int_{U_{\text{T}}}\int_{U_{0}}p(x,t|x_{0},0)\,\dd\mu(x_{0})\,\dd x,
\end{split}
\end{align}
where $\mu$ is the probability measure of the initial position of the searcher and $U_{0}$ is its support. It follows from \eqref{vpure} that
\begin{align}\label{laplacetype}
\lim_{t\to0+}t\ln\int_{U_{\text{T}}}\int_{U_{0}}p(x,t|x_{0},0)\,\dd\mu(x_{0})\,\dd x
=-\frac{\deuc^{2}({U_{0}},U_{\text{T}})}{4D},
\end{align}
where $\deuc({U_{0}},U_{\text{T}})$ is defined in \eqref{dset}.

To see why \eqref{laplacetype} holds, define the neighborhood of $U_{0}$,
\begin{align*}
B
:=\{x\in M:\inf_{y\in U_{0}}\deuc(x,y)<2\deuc(U_{0},U_{\text{T}})\}.
\end{align*}
Then, we decompose the target into $U_{\text{T}}=\{U_{\text{T}}\cap B\}\cup\{U_{\text{T}}\backslash B\}$ to obtain
\begin{align}
\begin{split}\label{i1i2}
&\int_{U_{\text{T}}}\int_{U_{0}}p(x,t|x_{0},0)\,\dd\mu(x_{0})\,\dd x\\
&\quad=\int_{U_{\text{T}}\cap B}\int_{U_{0}}p(x,t|x_{0},0)\,\dd\mu(x_{0})\,\dd x\\
&\quad\quad+\int_{U_{\text{T}}\backslash B}\int_{U_{0}}p(x,t|x_{0},0)\,\dd\mu(x_{0})\,\dd x
=:I_{1}+I_{2}.
\end{split}
\end{align}
To handle the first integral, $I_{1}$, let $\eps>0$ and note that \eqref{vpureS} holds uniformly for $x_{0},x$ in compact sets \cite{varadhan1967}. Hence, we may choose $t_{0}>0$ so that
\begin{align*}
e^{-\frac{\deuc^{2}(x_{0},x)+\eps}{4Dt}}
\le p(x,t|x_{0},0)
\le
e^{-\frac{\deuc^{2}(x_{0},x)-\eps}{4Dt}}
\end{align*}
for all $t\in(0,t_{0}]$, $x\in U_{\text{T}}\cap B$, and $x_{0}\in U_{0}$. Now, a straightforward application of the Laplace principle (or Varadhan's lemma \cite{dembo1998}) yields
\begin{align*}
&\lim_{t\to0+}t\ln\int_{U_{\text{T}}\cap B}\int_{U_{0}}e^{-(\deuc^{2}(x_{0},x)\pm\eps)/(4Dt)}\,\dd\mu(x_{0})\,\dd x\\
&\quad=-\frac{\deuc^{2}({U_{0}},U_{\text{T}})\pm\eps}{4D}.
\end{align*}
Since $\eps>0$ is arbitrary, we obtain
\begin{align*}
\lim_{t\to0+}t\ln I_{1}
=-\frac{\deuc^{2}({U_{0}},U_{\text{T}})}{4D}.
\end{align*}
It is straightforward to show that the second integral, $I_{2}$, in \eqref{i1i2} does not contribute to the limit, and so we obtain \eqref{laplacetype}. Therefore, \eqref{bfb} and \eqref{laplacetype} imply
\begin{align}\label{liminfp}
\liminf_{t\to0+}t\ln(1-S(t))\ge-\frac{\deuc^{2}({U_{0}},U_{\text{T}})}{4D}.
\end{align}

To bound $1-S(t)$ from above, let $\eps>0$ and define the set of all points in $M=\R^{d}$ that are more than distance $\eps$ from the target,
\begin{align*}
M_{\eps}
:=\{x\in M:\inf_{y\in U_{\text{T}}}\deuc(x,y)>\eps\}.
\end{align*}
Then decompose $1-S(t)$ into the case that the diffusion is either in $M_{\eps}$ or not in $M_{\eps}$ at time $t$,
\begin{align*}
1-S(t)
=\P(\tau<t,X(t)\in M_{\eps})+\P(\tau<t,X(t)\notin M_{\eps}).
\end{align*}
In the case that the diffusion is outside of $M_{\eps}$, we have
\begin{align*}
\P(\tau<t,X(t)\notin M_{\eps})
&\le\P(X(t)\notin M_{\eps})\\
&=\int_{M\backslash M_{\eps}}\int_{U_{0}}p(x,t|x_{0},0)\,\dd\mu(x_{0})\,\dd x,
\end{align*}
and it again follows from \eqref{vpure} that
\begin{align}\label{zzz}
\begin{split}
&\lim_{t\to0+}t\ln\int_{M\backslash M_{\eps}}\int_{U_{0}}p(x,t|x_{0},0)\,\dd\mu(x_{0})\,\dd x\\
&=-\frac{\deuc^{2}({U_{0}},M\backslash M_{\eps})}{4D}
\le-\frac{(\deuc({U_{0}},U_{\text{T}})-\eps)^{2}}{4D},
\end{split}
\end{align}
since
\begin{align*}
\deuc(U_{0},U_{\text{T}})
&\le \deuc(U_{0},M\backslash M_{\eps})+\deuc(M\backslash M_{\eps},U_{\text{T}})\\
&\le \deuc(U_{0},M\backslash M_{\eps})+\eps,
\end{align*}
by definition of $M_{\eps}$.

It remains to bound $\P(\tau<t,X(t)\in M_{\eps})$. Notice that $\P(\tau<t,X(t)\in M_{\eps})$ is the probability of paths which hit the target before time $t$ (since $\tau<t$) and then move distance $\eps$ away from the target (since $X(t)\in M_{\eps}$). The basic idea is that for small $t$, these paths are less likely than paths that hit the target before time $t$ and stay within an $\eps$ neighborhood of the target. To make this precise, let $s\in(0,t)$ and $z\in\partial U_{\text{T}}$ denote the respective time and position that the diffusion hits the target and use the strong Markov property to obtain
\begin{align}\label{pptt}
\begin{split}
&\P(\tau<t,X(t)\in M_{\eps})\\
&\quad=\int_{0}^{t}\int_{\partial U_{\text{T}}}\int_{M_{\eps}}p(x,t-s|z,0)\,\dd x\,\dd \nu_{s}(z)\,\dd\nu(s),
\end{split}
\end{align}
where $\nu$ denotes the probability measure for the time the diffusion hits the target and $\nu_{s}$ denotes the probability measure for the position that the diffusion hits the target conditioned upon hitting it at time $s$. Now if $s\in(0,t)$, then it follows from the definition of $M_{\eps}$ that
\begin{align*}
&\int_{\partial U_{\text{T}}}\int_{M_{\eps}}p(x,t-s|z,0)\,\dd x\,\dd \nu_{s}(z)\\
&\quad\le\P(\deuc(X(t-s),X(0))>\eps)\\
&\quad\le\P(\deuc(X(t),X(0))>\eps),
\end{align*}
and $\P(\deuc(X(t),X(0))>\eps)\to0$ as $t\to0+$. Similarly, if $s\in(0,t)$, then
\begin{align*}
&\int_{\partial U_{\text{T}}}\int_{M\backslash M_{\eps}}p(x,t-s|z,0)\,\dd x\,\dd \nu_{s}(z)\\
&\quad\ge\P(\deuc(X(t-s),X(0))<\eps)\\
&\quad\ge\P(\deuc(X(t),X(0))<\eps),
\end{align*}
and $\P(\deuc(X(t),X(0))<\eps)\to1$ as $t\to0+$. It follows then from \eqref{pptt} that
\begin{align*}
&\P(\tau<t,X(t)\in M_{\eps})\\
&\quad\le\int_{0}^{t}\int_{\partial U_{\text{T}}}\int_{M\backslash M_{\eps}}p(x,t-s|z,0)\,\dd x\,\dd \nu_{s}(z)\,\dd\nu(s)\\
&\quad=\P(\tau<t,X(t)\notin M_{\eps})
\le\P(X(t)\notin M_{\eps}),
\end{align*}
and we bounded the short time behavior of $\P(X(t)\notin M_{\eps})$ in \eqref{zzz}. We therefore obtain the upper bound
\begin{align*}
\limsup_{t\to0+}t\ln(1-S(t))\le-\frac{(\deuc({U_{0}},U_{\text{T}})-\eps)^{2}}{4D}.
\end{align*}
Using \eqref{liminfp} and the fact that $\eps>0$ is arbitrary completes the proof.
\end{proof}

\begin{proof}[Proof that \eqref{vv} implies \eqref{vvS}]
The proof follows along similar lines as the previous proof. 
\end{proof}

\begin{proof}[Proof that \eqref{vvnorris} implies \eqref{vvnorrisS}]
This can be proven along similar lines as the previous proof, though in this setup, the proof follows directly from Theorem 1.2 in \cite{norris1997}.
\end{proof}

\section{Partially absorbing boundary\label{partial}}

We now prove that the extreme statistics for a partially absorbing target versus a perfectly absorbing target are identical in the case of pure diffusion in a general class of $d$-dimensional spatial domains.

Let $\{X(t)\}_{t\ge0}$ denote the path of a searcher diffusing with diffusivity $D>0$ in a bounded domain $U\subset\R^{{d}}$ with reflecting boundaries and a partially absorbing target $U_{\text{T}}$ with reactivity $\kappa>0$ (assume $U$ is bounded, open, connected, and has a smooth boundary). Suppose the target $U_{\text{T}}\subset U$ is a finite, disjoint union of open balls (we took $U_{\text{T}}$ to be closed in the main text, but it is notationally convenient to take $U_{\text{T}}$ open in this setting). Hence, $X(t)$ satisfies the SDE,
\begin{align}\label{sdepab}
\begin{split}
\dd X
&=\sqrt{2D}\,\dd W+\nu(X)\,\dd L+\nu_{\text{T}}(X)\,\dd L_{\text{T}},
\end{split}
\end{align}
where $W(t)\in\R^{{d}}$ is a standard ${{d}}$-dimensional Brownian motion,
\begin{align}\label{iunf}
\nu:\partial U\to\R^{{d}},\quad\nu_{\text{T}}:\partial U_{\text{T}}\to \R^{{d}},
\end{align}
are the unit normal fields, both pointing into $U\backslash U_{\text{T}}$, and $L(t),L_{\text{T}}(t)$ are the local times of $X(t)$ on the boundaries of $U$ and $U_{\text{T}}$, respectively. The significance of the local time terms in \eqref{sdepab} is that they force $X(t)$ to reflect from the boundary of $U$ and the boundary of $U_{\text{T}}$. For simplicity, assume that the initial distribution of $X$ is a Dirac mass at a single point,
\begin{align*}
X(0)={x_{0}}\in U\backslash\overline{U_{\text{T}}}.
\end{align*}

The searcher is said to be absorbed at the partially absorbing target once its local time on the target surpasses an independent exponential random variable with rate $\kappa$. That is, the absorption time is
\begin{align*}
\tau_{\kappa}
:=\inf\{t>0:L_{\text{T}}(t)>\Sigma_{\kappa}\},
\end{align*}
where $\Sigma_{\kappa}\ge0$ is independent of $X(t)$ and satisfies
\begin{align*}
\P(\Sigma_{\kappa}>t)=e^{-\kappa t}.
\end{align*}
For technical reasons, it is convenient to continue to allow $X(t)$ to diffuse in $U\backslash U_{\text{T}}$ according to \eqref{sdepab} after the ``absorption time'' $\tau_{k}$. Notice that
\begin{align}\label{before}
\tau_{\kappa}\ge\tau:=\inf\{t>0:X(t)\in \overline{U_{\text{T}}}\}.
\end{align}
That is, the searcher is must reach the target before it can be absorbed at the target. Of course, $\tau_{\kappa}=\tau$ if the target is perfectly absorbing, $\kappa=+\infty$.

Define the survival probabilities $S_{\kappa}(t):=\P(\tau_{\kappa}>t)$ and $S(t):=\P(\tau>t)$. Then \eqref{before} implies
\begin{align*}
S_{\kappa}(t)
&=S(t)+\P(\tau_{\kappa}>t,\tau<t)\\
&=S(t)+(1-S(t))\P(\tau_{\kappa}>t\,|\,\tau<t).
\end{align*}
Therefore,
\begin{align}\label{bef}
\begin{split}
&\lim_{t\to0+}t\ln(1-S_{\kappa}(t))
=\lim_{t\to0+}t\ln(1-S(t))\\
&\qquad\qquad\qquad+\lim_{t\to0+}t\ln(1-\P(\tau_{\kappa}>t\,|\,\tau<t)).
\end{split}
\end{align}
To show that the asymptotic behavior of the extreme FPT is unaffected by the partial absorption, Theorem~\ref{theorem lg} implies that it remains to show that
\begin{align}\label{sss}
\lim_{t\to0+}t\ln(1-\P(\tau_{\kappa}>t\,|\,\tau<t))=0.
\end{align}

Using the definition of conditional probability and the strong Markov property gives
\begin{align*}
\P(\tau_{\kappa}>t\,|\,\tau<t)
\le\frac{\int_{0}^{t}\sup_{y\in\partial U_{\text{T}}}\P_{y}(\tau_{\kappa}>t-s)f(s)\,\dd s}
{1-S(t)},
\end{align*}
where $\P_{y}$ denotes the probability measure conditioned on $X(0)=y$ and $f(s)=-S'(s)$ is the density of $\tau$. At this point, assume that there exists a function $\Sb(t)$ satisfying the following three conditions,
\begin{align}
&\Sb(0)
=1,\label{c1}\\
&\sup_{y\in\partial U_{\text{T}}}\P_{y}(\tau_{\kappa}>t)
\le \Sb(t),\;\text{for $t$ sufficiently small},\label{c2}\\
&\Sb'(t)
\le-\lambda<0,\;\text{for $t$ sufficiently small}.\label{c3}
\end{align}
We will return to the question of the existence of such a function $\Sb$ in the subsection below.

Using \eqref{c1}-\eqref{c3} yields that for small $t$,
\begin{align*}
\int_{0}^{t}\sup_{y\in\partial U_{\text{T}}}\P_{y}(\tau_{\kappa}>t-s)f(s)\,\dd s
\le\int_{0}^{t}\Sb(t-s)f(s)\,\dd s\\
\quad\le1-S(t)-\lambda\int_{0}^{t}(1-S(s))\,\dd s,
\end{align*}
after integrating by parts. Therefore,
\begin{align*}
t\ln(1-\P(\tau_{\kappa}>t\,|\,\tau<t))
\ge
t\ln\Big(\frac{\lambda\int_{0}^{t}(1-S(s))\,\dd s}{1-S(t)}\Big)\\
=
t\ln\Big(\lambda\int_{0}^{t}(1-S(s))\,\dd s\Big)
-
t\ln(1-S(t)).
\end{align*}
Notice that
\begin{align}\label{twtw}
\lim_{t\to0+}t\ln(1-S(t))
=-\frac{\deuc^{2}({x_{0}},U_{\text{T}})}{4D},
\end{align}
and $\lim_{t\to0+}t\ln(1-\P(\tau_{\kappa}>t\,|\,\tau<t))\le0$. Hence, in order to verify \eqref{sss}, it remains to show that
\begin{align*}
\lim_{t\to0+}t\ln\Big(\lambda\int_{0}^{t}(1-S(s))\,\dd s\Big)
\ge-\frac{\deuc^{2}({x_{0}},U_{\text{T}})}{4D}.
\end{align*}

It follows from \eqref{twtw} that if $\eps>0$, then
\begin{align*}
1-S(t)\ge \exp\Big(-\frac{\deuc^{2}({x_{0}},U_{\text{T}})+\eps}{4Dt}\Big)
\end{align*}
for all $t$ sufficiently small. Therefore,
\begin{align*}
&\lim_{t\to0+}t\ln\lambda\int_{0}^{t}(1-S(s))\,\dd s\\
&\quad\ge\lim_{t\to0+}t\ln\lambda\int_{0}^{t}\exp\Big(-\frac{\deuc^{2}({x_{0}},U_{\text{T}})+\eps}{4Ds}\Big)\,\dd s\\
&\quad=-\frac{\deuc^{2}({x_{0}},U_{\text{T}})+\eps}{4D},
\end{align*}
Since $\eps>0$ is arbitrary, \eqref{twtw} is verified and we conclude from \eqref{bef} that
\begin{align*}
\lim_{t\to0+}t\ln(1-S_{\kappa}(t))=\lim_{t\to0+}t\ln(1-S(t)).
\end{align*}

\subsection{The existence of $\Sb$ in \eqref{c1}-\eqref{c3}\label{sectionexistence}}

\subsubsection{Symmetric case where $U_{\text{T}}\subset U$ are concentric balls}
We now show that there exists a function $\Sb(t)$ satisfying \eqref{c1}-\eqref{c3}. First, consider the special case of a symmetric problem where $U\subset\R^{{d}}$ is a ball and $U_{\text{T}}\subset U$ is a single ball located at the center of $U$. Specifically, if we denote the open ball of radius $r>0$ centered at $z\in\R^{{d}}$ by
\begin{align*}
B_{r}(z)
:=\{x\in\R^{{d}}:|x-z|<r\},
\end{align*}
then suppose
\begin{align}\label{sym}
U_{\text{T}}=B_{r}(0)\subset U=B_{R}(0),
\end{align}
for $0<r<R$. The survival probability conditioned on an initial location $X(0)=y\in\overline{U\backslash U_{\text{T}}}$,
\begin{align*}
\s(y,t)
=\P_{y}(\tau_{\kappa}>t),
\end{align*}
satisfies the backward Fokker-Planck equation \cite{gardiner2009, pavliotis2014},
\begin{align*}
\tfrac{\partial}{\partial t}\s
&=D\Delta \s,\quad y\in U\backslash\overline{U_{\text{T}}},\\
\tfrac{\partial}{\partial \nu}\s
&=0,\quad y\in\partial U,\\
D\tfrac{\partial}{\partial \nu_{\text{T}}}\s
&=\kappa \s,\quad y\in\partial U_{\text{T}},\\
\s
&=1,\quad t=0.
\end{align*}
where $\tfrac{\partial}{\partial \nu}$ and $\tfrac{\partial}{\partial \nu_{\text{T}}}$ denote derivatives with respect to the inward unit normal fields \eqref{iunf}.

Define the survival probability conditioned on starting on the target,
\begin{align}\label{s0}
S_{0}(t)
:=\s(y,t)\quad\text{for }y\in\partial U_{\text{T}}.
\end{align}
Note that \eqref{s0} is the same for any choice of $y\in\partial U_{\text{T}}$ by symmetry.
It was shown in Section IIIB of \cite{lawley2019imp} that there exists a $\lambda_{0}>0$ so that
\begin{align*}
-S_{0}'(t)\ge\lambda_{0} S(t)\quad\text{for all }t>0.
\end{align*}
 Hence, \eqref{c1}-\eqref{c3} are satisfied with $\Sb(t)=e^{-\lambda_{0}t}$ in the special case that $U_{\text{T}}\subset U$ are concentric balls with respective radii $r<R$.
 
\subsubsection{General case}

We now extend to the case that $U\subset\R^{{d}}$ is a bounded domain and the target is a finite, disjoint union of  balls,
\begin{align*}
U_{\text{T}}=\cup_{k=1}^{K}B_{r_{k}}(z_{k})\subset U.
\end{align*}
Let $y\in\partial U_{\text{T}}$. Without loss of generality, suppose $y\in\partial B_{r_{1}}(z_{1})$.

There exists a $\delta>r_{1}>0$ so that $B_{\delta}(z_{1})\subset U$ and $B_{\delta}(z_{1})\cap U_{\text{T}}=\overline{B_{\delta}(z_{1})}$. Then for each $t>0$, we have that
\begin{align}\label{bas}
\P_{y}(\tau_{\kappa}<t)
=\P_{y}(\tau_{\kappa}<t,\tau_{\delta}>t)+\P_{y}(\tau_{\kappa}<t,\tau_{\delta}<t),
\end{align}
where $\tau_{\delta}$ is the first time the searcher escapes $B_{\delta}(z_{1})$.

Now, for the symmetric problem in \eqref{sym} with $r=r_{1}$ and $R=2\delta$, let $\tau_{\kappa}^{\text{sym}}$ and $\tau_{\delta}^{\text{sym}}$ be the absorption time at $U_{\text{T}}=B_{r_{1}}(0)$ and the hitting time to $\partial B_{\delta}(0)$, respectively. It is immediate that if $|y_{0}|=r_{1}$, then
\begin{align}\label{eqeq}
\P_{y}(\tau_{\kappa}<t,\tau_{\delta}>t)
=\P_{y_{0}}(\tau_{\kappa}^{\text{sym}}<t,\tau_{\delta}^{\text{sym}}>t).
\end{align}
Hence, dividing \eqref{bas} by $\P_{y_{0}}(\tau_{\kappa}^{\text{sym}}<t)$ for $|y_{0}|=r_{1}$ yields
\begin{align}\label{t21}
\begin{split}
\frac{\P_{y}(\tau_{\kappa}<t)}{\P_{y_{0}}(\tau_{\kappa}^{\text{sym}}<t)}
=\frac{\P_{y}(\tau_{\kappa}<t,\tau_{\delta}>t)}{\P_{y_{0}}(\tau_{\kappa}^{\text{sym}}<t)}
+\frac{\P_{y}(\tau_{\kappa}<t,\tau_{\delta}<t)}{\P_{y_{0}}(\tau_{\kappa}^{\text{sym}}<t)}.
\end{split}
\end{align}
Rearranging \eqref{t21} and using \eqref{eqeq} yields
\begin{align*}
\frac{\P_{y}(\tau_{\kappa}<t)}{\P_{y_{0}}(\tau_{\kappa}^{\text{sym}}<t)}
&=\Big(1-\frac{\P_{y_{0}}(\tau_{\kappa}^{\text{sym}}<t,\tau_{\delta}^{\text{sym}}<t)}{\P_{y_{0}}(\tau_{\kappa}^{\text{sym}}<t,\tau_{\delta}^{\text{sym}}>t)}\Big)^{-1}\\
&\quad+\frac{\P_{y}(\tau_{\kappa}<t,\tau_{\delta}<t)}{\P_{y_{0}}(\tau_{\kappa}^{\text{sym}}<t)}.
\end{align*}
We claim that
\begin{align}\label{claim}
\begin{split}
\lim_{t\to0}\frac{\P_{y_{0}}(\tau_{\kappa}^{\text{sym}}<t,\tau_{\delta}^{\text{sym}}<t)}{\P_{y_{0}}(\tau_{\kappa}^{\text{sym}}<t,\tau_{\delta}^{\text{sym}}>t)}
=\lim_{t\to0}\frac{\P_{y}(\tau_{\kappa}<t,\tau_{\delta}<t)}{\P_{y_{0}}(\tau_{\kappa}^{\text{sym}}<t)}
=0,
\end{split}
\end{align}
and thus
\begin{align}\label{thus}
\lim_{t\to0+}\frac{\P_{y}(\tau_{\kappa}<t)}{\P_{y_{0}}(\tau_{\kappa}^{\text{sym}}<t)}=1.
\end{align}

To see why \eqref{claim} holds, note first that
\begin{align*}
&\max\Big\{\frac{\P_{y}(\tau_{\kappa}<t,\tau_{\delta}<t)}{\P_{y_{0}}(\tau_{\kappa}^{\text{sym}}<t)},\frac{\P_{y_{0}}(\tau_{\kappa}^{\text{sym}}<t,\tau_{\delta}^{\text{sym}}<t)}{\P_{y_{0}}(\tau_{\kappa}^{\text{sym}}<t,\tau_{\delta}^{\text{sym}}>t)}\Big\}\\
&\quad\le\frac{\P_{y_{0}}(\tau_{\delta}^{\text{sym}}<t)}{\P_{y_{0}}(\tau_{\kappa}^{\text{sym}}<t,\tau_{\delta}^{\text{sym}}>t)}.
\end{align*}
Now it follows from Varadhan's formula \cite{varadhan1967} that 
\begin{align*}
\P_{y_{0}}(\tau_{\delta}^{\text{sym}}<t)\le e^{-(\delta-r_{1})^{2}/(5Dt)}
\end{align*}
for sufficiently small $t$. Further, we established above that there exists a $\lambda_{0}>0$ so that
\begin{align}\label{other}
\P_{y_{0}}(\tau_{\kappa}^{\text{sym}}<t)\ge1-e^{-\lambda_{0}t}
\end{align}
for sufficiently small $t$. Hence,
\begin{align*}
&\frac{\P_{y_{0}}(\tau_{\delta}^{\text{sym}}<t)}{\P_{y_{0}}(\tau_{\kappa}^{\text{sym}}<t,\tau_{\delta}^{\text{sym}}>t)}\\
&\quad=\frac{\P_{y_{0}}(\tau_{\delta}^{\text{sym}}<t)}{\P_{y_{0}}(\tau_{\kappa}^{\text{sym}}<t)-\P_{y_{0}}(\tau_{\kappa}^{\text{sym}}<t,\tau_{\delta}^{\text{sym}}<t)}\\
&\quad\le\frac{e^{-(\delta-r_{1})^{2}/(5Dt)}}{1-e^{-\lambda_{0}t}-e^{-(\delta-r_{1})^{2}/(5Dt)}}
\end{align*}
for sufficiently small $t$. Taking $t\to0+$ thus verifies \eqref{claim}.

We claim that 
\begin{align}\label{want}
\P_{y}(\tau_{\kappa}>t)\le e^{-(\lambda_{0}/2)t}\quad\text{for sufficiently small }t.
\end{align}
To see this, note that if \eqref{want} is false, then using \eqref{thus}, \eqref{other}, and L'Hospital's rule yields
\begin{align*}
1=\lim_{t\to0+}\frac{\P_{y}(\tau_{\kappa}<t)}{\P_{y_{0}}(\tau_{\kappa}^{\text{sym}}<t)}
\le\lim_{t\to0+}\frac{1-e^{-(\lambda_{0}/2)t}}{1-e^{-\lambda_{0}t}}
=\frac{1}{2},
\end{align*}
which is absurd.  Hence, \eqref{c1}-\eqref{c3} are satisfied with $\Sb(t)=e^{-(\lambda_{0}/2)t}$.


\bibliography{library.bib}

\begin{thebibliography}{51}%
\makeatletter
\providecommand \@ifxundefined [1]{%
 \@ifx{#1\undefined}
}%
\providecommand \@ifnum [1]{%
 \ifnum #1\expandafter \@firstoftwo
 \else \expandafter \@secondoftwo
 \fi
}%
\providecommand \@ifx [1]{%
 \ifx #1\expandafter \@firstoftwo
 \else \expandafter \@secondoftwo
 \fi
}%
\providecommand \natexlab [1]{#1}%
\providecommand \enquote  [1]{``#1''}%
\providecommand \bibnamefont  [1]{#1}%
\providecommand \bibfnamefont [1]{#1}%
\providecommand \citenamefont [1]{#1}%
\providecommand \href@noop [0]{\@secondoftwo}%
\providecommand \href [0]{\begingroup \@sanitize@url \@href}%
\providecommand \@href[1]{\@@startlink{#1}\@@href}%
\providecommand \@@href[1]{\endgroup#1\@@endlink}%
\providecommand \@sanitize@url [0]{\catcode `\\12\catcode `\$12\catcode
  `\&12\catcode `\#12\catcode `\^12\catcode `\_12\catcode `\%12\relax}%
\providecommand \@@startlink[1]{}%
\providecommand \@@endlink[0]{}%
\providecommand \url  [0]{\begingroup\@sanitize@url \@url }%
\providecommand \@url [1]{\endgroup\@href {#1}{\urlprefix }}%
\providecommand \urlprefix  [0]{URL }%
\providecommand \Eprint [0]{\href }%
\providecommand \doibase [0]{http://dx.doi.org/}%
\providecommand \selectlanguage [0]{\@gobble}%
\providecommand \bibinfo  [0]{\@secondoftwo}%
\providecommand \bibfield  [0]{\@secondoftwo}%
\providecommand \translation [1]{[#1]}%
\providecommand \BibitemOpen [0]{}%
\providecommand \bibitemStop [0]{}%
\providecommand \bibitemNoStop [0]{.\EOS\space}%
\providecommand \EOS [0]{\spacefactor3000\relax}%
\providecommand \BibitemShut  [1]{\csname bibitem#1\endcsname}%
\let\auto@bib@innerbib\@empty
\bibitem [{\citenamefont {Redner}(2001)}]{redner2001}%
  \BibitemOpen
  \bibfield  {author} {\bibinfo {author} {\bibfnamefont {S.}~\bibnamefont
  {Redner}},\ }\href@noop {} {\emph {\bibinfo {title} {A guide to first-passage
  processes}}}\ (\bibinfo  {publisher} {Cambridge University Press},\ \bibinfo
  {year} {2001})\BibitemShut {NoStop}%
\bibitem [{\citenamefont {Helmholtz}(1860)}]{helmholtz1860}%
  \BibitemOpen
  \bibfield  {author} {\bibinfo {author} {\bibfnamefont {H.}~\bibnamefont
  {Helmholtz}},\ }\href@noop {} {\bibfield  {journal} {\bibinfo  {journal}
  {Journal f{\"u}r die reine und angewandte Mathematik}\ }\textbf {\bibinfo
  {volume} {57}},\ \bibinfo {pages} {1} (\bibinfo {year} {1860})}\BibitemShut
  {NoStop}%
\bibitem [{\citenamefont {Rayleigh}(1945)}]{rayleigh1945}%
  \BibitemOpen
  \bibfield  {author} {\bibinfo {author} {\bibfnamefont {J.~W.~S.}\
  \bibnamefont {Rayleigh}},\ }\href@noop {} {\emph {\bibinfo {title} {The
  theory of sound}}}\ (\bibinfo  {publisher} {Dover},\ \bibinfo {year}
  {1945})\BibitemShut {NoStop}%
\bibitem [{\citenamefont {B{\'e}nichou}\ and\ \citenamefont
  {Voituriez}(2008)}]{benichou2008}%
  \BibitemOpen
  \bibfield  {author} {\bibinfo {author} {\bibfnamefont {O.}~\bibnamefont
  {B{\'e}nichou}}\ and\ \bibinfo {author} {\bibfnamefont {R.}~\bibnamefont
  {Voituriez}},\ }\href@noop {} {\bibfield  {journal} {\bibinfo  {journal}
  {Phys Rev Lett}\ }\textbf {\bibinfo {volume} {100}},\ \bibinfo {pages}
  {168105} (\bibinfo {year} {2008})}\BibitemShut {NoStop}%
\bibitem [{\citenamefont {Reingruber}\ and\ \citenamefont
  {Holcman}(2009)}]{Reingruber2009}%
  \BibitemOpen
  \bibfield  {author} {\bibinfo {author} {\bibfnamefont {J.}~\bibnamefont
  {Reingruber}}\ and\ \bibinfo {author} {\bibfnamefont {D.}~\bibnamefont
  {Holcman}},\ }\href@noop {} {\bibfield  {journal} {\bibinfo  {journal} {Phys
  Rev Lett}\ }\textbf {\bibinfo {volume} {103}},\ \bibinfo {pages} {148102}
  (\bibinfo {year} {2009})}\BibitemShut {NoStop}%
\bibitem [{\citenamefont {Benichou}\ \emph {et~al.}(2010)\citenamefont
  {Benichou}, \citenamefont {Grebenkov}, \citenamefont {Levitz}, \citenamefont
  {Loverdo},\ and\ \citenamefont {Voituriez}}]{benichou2010}%
  \BibitemOpen
  \bibfield  {author} {\bibinfo {author} {\bibfnamefont {O.}~\bibnamefont
  {Benichou}}, \bibinfo {author} {\bibfnamefont {D.}~\bibnamefont {Grebenkov}},
  \bibinfo {author} {\bibfnamefont {P.}~\bibnamefont {Levitz}}, \bibinfo
  {author} {\bibfnamefont {C.}~\bibnamefont {Loverdo}}, \ and\ \bibinfo
  {author} {\bibfnamefont {R.}~\bibnamefont {Voituriez}},\ }\href {\doibase
  10.1103/PhysRevLett.105.150606} {\bibfield  {journal} {\bibinfo  {journal}
  {Phys Rev Lett}\ }\textbf {\bibinfo {volume} {105}},\ \bibinfo {pages}
  {150606} (\bibinfo {year} {2010})}\BibitemShut {NoStop}%
\bibitem [{\citenamefont {Holcman}\ and\ \citenamefont
  {Schuss}(2014{\natexlab{a}})}]{holcman2014time}%
  \BibitemOpen
  \bibfield  {author} {\bibinfo {author} {\bibfnamefont {D.}~\bibnamefont
  {Holcman}}\ and\ \bibinfo {author} {\bibfnamefont {Z.}~\bibnamefont
  {Schuss}},\ }\href@noop {} {\bibfield  {journal} {\bibinfo  {journal} {J Phys
  A}\ }\textbf {\bibinfo {volume} {47}},\ \bibinfo {pages} {173001} (\bibinfo
  {year} {2014}{\natexlab{a}})}\BibitemShut {NoStop}%
\bibitem [{\citenamefont {Holcman}\ and\ \citenamefont
  {Schuss}(2014{\natexlab{b}})}]{holcman2014}%
  \BibitemOpen
  \bibfield  {author} {\bibinfo {author} {\bibfnamefont {D.}~\bibnamefont
  {Holcman}}\ and\ \bibinfo {author} {\bibfnamefont {Z.}~\bibnamefont
  {Schuss}},\ }\href {\doibase 10.1137/120898395} {\bibfield  {journal}
  {\bibinfo  {journal} {{SIAM} Rev}\ }\textbf {\bibinfo {volume} {56}},\
  \bibinfo {pages} {213} (\bibinfo {year} {2014}{\natexlab{b}})}\BibitemShut
  {NoStop}%
\bibitem [{\citenamefont {Calandre}\ \emph {et~al.}(2014)\citenamefont
  {Calandre}, \citenamefont {B{\'e}nichou},\ and\ \citenamefont
  {Voituriez}}]{calandre2014}%
  \BibitemOpen
  \bibfield  {author} {\bibinfo {author} {\bibfnamefont {T.}~\bibnamefont
  {Calandre}}, \bibinfo {author} {\bibfnamefont {O.}~\bibnamefont
  {B{\'e}nichou}}, \ and\ \bibinfo {author} {\bibfnamefont {R.}~\bibnamefont
  {Voituriez}},\ }\href@noop {} {\bibfield  {journal} {\bibinfo  {journal}
  {Phys Rev Lett}\ }\textbf {\bibinfo {volume} {112}},\ \bibinfo {pages}
  {230601} (\bibinfo {year} {2014})}\BibitemShut {NoStop}%
\bibitem [{\citenamefont {Vaccario}\ \emph {et~al.}(2015)\citenamefont
  {Vaccario}, \citenamefont {Antoine},\ and\ \citenamefont
  {Talbot}}]{vaccario2015}%
  \BibitemOpen
  \bibfield  {author} {\bibinfo {author} {\bibfnamefont {G.}~\bibnamefont
  {Vaccario}}, \bibinfo {author} {\bibfnamefont {C.}~\bibnamefont {Antoine}}, \
  and\ \bibinfo {author} {\bibfnamefont {J.}~\bibnamefont {Talbot}},\
  }\href@noop {} {\bibfield  {journal} {\bibinfo  {journal} {Phys Rev Lett}\
  }\textbf {\bibinfo {volume} {115}},\ \bibinfo {pages} {240601} (\bibinfo
  {year} {2015})}\BibitemShut {NoStop}%
\bibitem [{\citenamefont {Grebenkov}(2016)}]{grebenkov2016}%
  \BibitemOpen
  \bibfield  {author} {\bibinfo {author} {\bibfnamefont {D.~S.}\ \bibnamefont
  {Grebenkov}},\ }\href@noop {} {\bibfield  {journal} {\bibinfo  {journal}
  {Phys Rev Lett}\ }\textbf {\bibinfo {volume} {117}},\ \bibinfo {pages}
  {260201} (\bibinfo {year} {2016})}\BibitemShut {NoStop}%
\bibitem [{\citenamefont {Newby}\ and\ \citenamefont
  {Allard}(2016)}]{newby2016}%
  \BibitemOpen
  \bibfield  {author} {\bibinfo {author} {\bibfnamefont {J.}~\bibnamefont
  {Newby}}\ and\ \bibinfo {author} {\bibfnamefont {J.}~\bibnamefont {Allard}},\
  }\href@noop {} {\bibfield  {journal} {\bibinfo  {journal} {Phys Rev Lett}\
  }\textbf {\bibinfo {volume} {116}},\ \bibinfo {pages} {128101} (\bibinfo
  {year} {2016})}\BibitemShut {NoStop}%
\bibitem [{\citenamefont {Lindsay}\ \emph {et~al.}(2017)\citenamefont
  {Lindsay}, \citenamefont {Bernoff},\ and\ \citenamefont
  {Ward}}]{lindsay2017}%
  \BibitemOpen
  \bibfield  {author} {\bibinfo {author} {\bibfnamefont {A.~E.}\ \bibnamefont
  {Lindsay}}, \bibinfo {author} {\bibfnamefont {A.~J.}\ \bibnamefont
  {Bernoff}}, \ and\ \bibinfo {author} {\bibfnamefont {M.~J.}\ \bibnamefont
  {Ward}},\ }\href@noop {} {\bibfield  {journal} {\bibinfo  {journal}
  {Multiscale Model Simul}\ }\textbf {\bibinfo {volume} {15}},\ \bibinfo
  {pages} {74} (\bibinfo {year} {2017})}\BibitemShut {NoStop}%
\bibitem [{\citenamefont {Basnayake}\ \emph
  {et~al.}(2019{\natexlab{a}})\citenamefont {Basnayake}, \citenamefont
  {Schuss},\ and\ \citenamefont {Holcman}}]{basnayake2019}%
  \BibitemOpen
  \bibfield  {author} {\bibinfo {author} {\bibfnamefont {K.}~\bibnamefont
  {Basnayake}}, \bibinfo {author} {\bibfnamefont {Z.}~\bibnamefont {Schuss}}, \
  and\ \bibinfo {author} {\bibfnamefont {D.}~\bibnamefont {Holcman}},\
  }\href@noop {} {\bibfield  {journal} {\bibinfo  {journal} {J Nonlinear Sci}\
  }\textbf {\bibinfo {volume} {29}},\ \bibinfo {pages} {461} (\bibinfo {year}
  {2019}{\natexlab{a}})}\BibitemShut {NoStop}%
\bibitem [{\citenamefont {Schuss}\ \emph {et~al.}(2019)\citenamefont {Schuss},
  \citenamefont {Basnayake},\ and\ \citenamefont {Holcman}}]{schuss2019}%
  \BibitemOpen
  \bibfield  {author} {\bibinfo {author} {\bibfnamefont {Z.}~\bibnamefont
  {Schuss}}, \bibinfo {author} {\bibfnamefont {K.}~\bibnamefont {Basnayake}}, \
  and\ \bibinfo {author} {\bibfnamefont {D.}~\bibnamefont {Holcman}},\ }\href
  {\doibase 10.1016/j.plrev.2019.01.001} {\bibfield  {journal} {\bibinfo
  {journal} {Physics of Life Reviews}\ } (\bibinfo {year} {2019}),\
  10.1016/j.plrev.2019.01.001}\BibitemShut {NoStop}%
\bibitem [{\citenamefont {Coombs}(2019)}]{coombs2019}%
  \BibitemOpen
  \bibfield  {author} {\bibinfo {author} {\bibfnamefont {D.}~\bibnamefont
  {Coombs}},\ }\href {\doibase 10.1016/j.plrev.2019.03.002} {\bibfield
  {journal} {\bibinfo  {journal} {Physics of Life Reviews}\ }\textbf {\bibinfo
  {volume} {28}},\ \bibinfo {pages} {92} (\bibinfo {year} {2019})}\BibitemShut
  {NoStop}%
\bibitem [{\citenamefont {Redner}\ and\ \citenamefont
  {Meerson}(2019)}]{redner2019}%
  \BibitemOpen
  \bibfield  {author} {\bibinfo {author} {\bibfnamefont {S.}~\bibnamefont
  {Redner}}\ and\ \bibinfo {author} {\bibfnamefont {B.}~\bibnamefont
  {Meerson}},\ }\href {\doibase 10.1016/j.plrev.2019.01.020} {\bibfield
  {journal} {\bibinfo  {journal} {Physics of Life Reviews}\ }\textbf {\bibinfo
  {volume} {28}},\ \bibinfo {pages} {80} (\bibinfo {year} {2019})}\BibitemShut
  {NoStop}%
\bibitem [{\citenamefont {Sokolov}(2019)}]{sokolov2019}%
  \BibitemOpen
  \bibfield  {author} {\bibinfo {author} {\bibfnamefont {I.~M.}\ \bibnamefont
  {Sokolov}},\ }\href {\doibase 10.1016/j.plrev.2019.03.003} {\bibfield
  {journal} {\bibinfo  {journal} {Physics of Life Reviews}\ }\textbf {\bibinfo
  {volume} {28}},\ \bibinfo {pages} {88} (\bibinfo {year} {2019})}\BibitemShut
  {NoStop}%
\bibitem [{\citenamefont {Rusakov}\ and\ \citenamefont
  {Savtchenko}(2019)}]{rusakov2019}%
  \BibitemOpen
  \bibfield  {author} {\bibinfo {author} {\bibfnamefont {D.~A.}\ \bibnamefont
  {Rusakov}}\ and\ \bibinfo {author} {\bibfnamefont {L.~P.}\ \bibnamefont
  {Savtchenko}},\ }\href {\doibase 10.1016/j.plrev.2019.02.001} {\bibfield
  {journal} {\bibinfo  {journal} {Physics of Life Reviews}\ }\textbf {\bibinfo
  {volume} {28}},\ \bibinfo {pages} {85} (\bibinfo {year} {2019})}\BibitemShut
  {NoStop}%
\bibitem [{\citenamefont {Martyushev}(2019)}]{martyushev2019}%
  \BibitemOpen
  \bibfield  {author} {\bibinfo {author} {\bibfnamefont {L.~M.}\ \bibnamefont
  {Martyushev}},\ }\href {\doibase 10.1016/j.plrev.2019.02.002} {\bibfield
  {journal} {\bibinfo  {journal} {Physics of Life Reviews}\ }\textbf {\bibinfo
  {volume} {28}},\ \bibinfo {pages} {83} (\bibinfo {year} {2019})}\BibitemShut
  {NoStop}%
\bibitem [{\citenamefont {Tamm}(2019)}]{tamm2019}%
  \BibitemOpen
  \bibfield  {author} {\bibinfo {author} {\bibfnamefont {M.~V.}\ \bibnamefont
  {Tamm}},\ }\href {\doibase 10.1016/j.plrev.2019.03.001} {\bibfield  {journal}
  {\bibinfo  {journal} {Physics of Life Reviews}\ }\textbf {\bibinfo {volume}
  {28}},\ \bibinfo {pages} {94} (\bibinfo {year} {2019})}\BibitemShut {NoStop}%
\bibitem [{\citenamefont {Basnayake}\ and\ \citenamefont
  {Holcman}(2019)}]{basnayake2019c}%
  \BibitemOpen
  \bibfield  {author} {\bibinfo {author} {\bibfnamefont {K.}~\bibnamefont
  {Basnayake}}\ and\ \bibinfo {author} {\bibfnamefont {D.}~\bibnamefont
  {Holcman}},\ }\href {\doibase 10.1016/j.plrev.2019.03.017} {\bibfield
  {journal} {\bibinfo  {journal} {Physics of Life Reviews}\ }\textbf {\bibinfo
  {volume} {28}},\ \bibinfo {pages} {96} (\bibinfo {year} {2019})}\BibitemShut
  {NoStop}%
\bibitem [{\citenamefont {Basnayake}\ \emph {et~al.}(2018)\citenamefont
  {Basnayake}, \citenamefont {Hubl}, \citenamefont {Schuss},\ and\
  \citenamefont {Holcman}}]{basnayake_extreme_2018}%
  \BibitemOpen
  \bibfield  {author} {\bibinfo {author} {\bibfnamefont {K.}~\bibnamefont
  {Basnayake}}, \bibinfo {author} {\bibfnamefont {A.}~\bibnamefont {Hubl}},
  \bibinfo {author} {\bibfnamefont {Z.}~\bibnamefont {Schuss}}, \ and\ \bibinfo
  {author} {\bibfnamefont {D.}~\bibnamefont {Holcman}},\ }\href {\doibase
  10.1016/j.physleta.2018.09.040} {\bibfield  {journal} {\bibinfo  {journal}
  {Physics Letters A}\ }\textbf {\bibinfo {volume} {382}},\ \bibinfo {pages}
  {3449} (\bibinfo {year} {2018})}\BibitemShut {NoStop}%
\bibitem [{\citenamefont {Reynaud}\ \emph {et~al.}(2015)\citenamefont
  {Reynaud}, \citenamefont {Schuss}, \citenamefont {Rouach},\ and\
  \citenamefont {Holcman}}]{reynaud2015}%
  \BibitemOpen
  \bibfield  {author} {\bibinfo {author} {\bibfnamefont {K.}~\bibnamefont
  {Reynaud}}, \bibinfo {author} {\bibfnamefont {Z.}~\bibnamefont {Schuss}},
  \bibinfo {author} {\bibfnamefont {N.}~\bibnamefont {Rouach}}, \ and\ \bibinfo
  {author} {\bibfnamefont {D.}~\bibnamefont {Holcman}},\ }\href {\doibase
  10.1080/19420889.2015.1017156} {\bibfield  {journal} {\bibinfo  {journal}
  {Communicative \& Integrative Biology}\ }\textbf {\bibinfo {volume} {8}},\
  \bibinfo {pages} {e1017156} (\bibinfo {year} {2015})}\BibitemShut {NoStop}%
\bibitem [{\citenamefont {Basnayake}\ \emph
  {et~al.}(2019{\natexlab{b}})\citenamefont {Basnayake}, \citenamefont
  {Mazaud}, \citenamefont {Bemelmans}, \citenamefont {Rouach}, \citenamefont
  {Korkotian},\ and\ \citenamefont {Holcman}}]{basnayake2019fast}%
  \BibitemOpen
  \bibfield  {author} {\bibinfo {author} {\bibfnamefont {K.}~\bibnamefont
  {Basnayake}}, \bibinfo {author} {\bibfnamefont {D.}~\bibnamefont {Mazaud}},
  \bibinfo {author} {\bibfnamefont {A.}~\bibnamefont {Bemelmans}}, \bibinfo
  {author} {\bibfnamefont {N.}~\bibnamefont {Rouach}}, \bibinfo {author}
  {\bibfnamefont {E.}~\bibnamefont {Korkotian}}, \ and\ \bibinfo {author}
  {\bibfnamefont {D.}~\bibnamefont {Holcman}},\ }\href {\doibase
  10.1371/journal.pbio.2006202} {\bibfield  {journal} {\bibinfo  {journal}
  {PLOS Biology}\ }\textbf {\bibinfo {volume} {17}},\ \bibinfo {pages}
  {e2006202} (\bibinfo {year} {2019}{\natexlab{b}})}\BibitemShut {NoStop}%
\bibitem [{\citenamefont {Guerrier}\ and\ \citenamefont
  {Holcman}(2018)}]{guerrier2018}%
  \BibitemOpen
  \bibfield  {author} {\bibinfo {author} {\bibfnamefont {C.}~\bibnamefont
  {Guerrier}}\ and\ \bibinfo {author} {\bibfnamefont {D.}~\bibnamefont
  {Holcman}},\ }\href {\doibase 10.3389/fnsyn.2018.00023} {\bibfield  {journal}
  {\bibinfo  {journal} {Frontiers in Synaptic Neuroscience}\ }\textbf {\bibinfo
  {volume} {10}} (\bibinfo {year} {2018}),\
  10.3389/fnsyn.2018.00023}\BibitemShut {NoStop}%
\bibitem [{\citenamefont {Meerson}\ and\ \citenamefont
  {Redner}(2015)}]{meerson2015}%
  \BibitemOpen
  \bibfield  {author} {\bibinfo {author} {\bibfnamefont {B.}~\bibnamefont
  {Meerson}}\ and\ \bibinfo {author} {\bibfnamefont {S.}~\bibnamefont
  {Redner}},\ }\href@noop {} {\bibfield  {journal} {\bibinfo  {journal} {Phys
  Rev Lett}\ }\textbf {\bibinfo {volume} {114}},\ \bibinfo {pages} {198101}
  (\bibinfo {year} {2015})}\BibitemShut {NoStop}%
\bibitem [{\citenamefont {Gumbel}(1962)}]{gumbel2012}%
  \BibitemOpen
  \bibfield  {author} {\bibinfo {author} {\bibfnamefont {E.~J.}\ \bibnamefont
  {Gumbel}},\ }\href@noop {} {\emph {\bibinfo {title} {Statistics of
  extremes}}}\ (\bibinfo  {publisher} {Columbia University Press},\ \bibinfo
  {year} {1962})\BibitemShut {NoStop}%
\bibitem [{\citenamefont {Weiss}\ \emph {et~al.}(1983)\citenamefont {Weiss},
  \citenamefont {Shuler},\ and\ \citenamefont {Lindenberg}}]{weiss1983}%
  \BibitemOpen
  \bibfield  {author} {\bibinfo {author} {\bibfnamefont {G.~H.}\ \bibnamefont
  {Weiss}}, \bibinfo {author} {\bibfnamefont {K.~E.}\ \bibnamefont {Shuler}}, \
  and\ \bibinfo {author} {\bibfnamefont {K.}~\bibnamefont {Lindenberg}},\
  }\href@noop {} {\bibfield  {journal} {\bibinfo  {journal} {J Stat Phys}\
  }\textbf {\bibinfo {volume} {31}},\ \bibinfo {pages} {255} (\bibinfo {year}
  {1983})}\BibitemShut {NoStop}%
\bibitem [{\citenamefont {Yuste}\ and\ \citenamefont
  {Lindenberg}(1996)}]{yuste1996}%
  \BibitemOpen
  \bibfield  {author} {\bibinfo {author} {\bibfnamefont {S.~B.}\ \bibnamefont
  {Yuste}}\ and\ \bibinfo {author} {\bibfnamefont {K.}~\bibnamefont
  {Lindenberg}},\ }\href@noop {} {\bibfield  {journal} {\bibinfo  {journal} {J
  Stat Phys}\ }\textbf {\bibinfo {volume} {85}},\ \bibinfo {pages} {501}
  (\bibinfo {year} {1996})}\BibitemShut {NoStop}%
\bibitem [{\citenamefont {Yuste}\ and\ \citenamefont
  {Acedo}(2000)}]{yuste2000}%
  \BibitemOpen
  \bibfield  {author} {\bibinfo {author} {\bibfnamefont {S.}~\bibnamefont
  {Yuste}}\ and\ \bibinfo {author} {\bibfnamefont {L.}~\bibnamefont {Acedo}},\
  }\href@noop {} {\bibfield  {journal} {\bibinfo  {journal} {J Phys A}\
  }\textbf {\bibinfo {volume} {33}},\ \bibinfo {pages} {507} (\bibinfo {year}
  {2000})}\BibitemShut {NoStop}%
\bibitem [{\citenamefont {Yuste}\ \emph {et~al.}(2001)\citenamefont {Yuste},
  \citenamefont {Acedo},\ and\ \citenamefont {Lindenberg}}]{yuste2001}%
  \BibitemOpen
  \bibfield  {author} {\bibinfo {author} {\bibfnamefont {S.~B.}\ \bibnamefont
  {Yuste}}, \bibinfo {author} {\bibfnamefont {L.}~\bibnamefont {Acedo}}, \ and\
  \bibinfo {author} {\bibfnamefont {K.}~\bibnamefont {Lindenberg}},\
  }\href@noop {} {\bibfield  {journal} {\bibinfo  {journal} {Phys Rev E}\
  }\textbf {\bibinfo {volume} {64}},\ \bibinfo {pages} {052102} (\bibinfo
  {year} {2001})}\BibitemShut {NoStop}%
\bibitem [{\citenamefont {Redner}\ and\ \citenamefont
  {Meerson}(2014)}]{redner2014}%
  \BibitemOpen
  \bibfield  {author} {\bibinfo {author} {\bibfnamefont {S.}~\bibnamefont
  {Redner}}\ and\ \bibinfo {author} {\bibfnamefont {B.}~\bibnamefont
  {Meerson}},\ }\href@noop {} {\bibfield  {journal} {\bibinfo  {journal} {J
  Stat Mech}\ }\textbf {\bibinfo {volume} {2014}},\ \bibinfo {pages} {P06019}
  (\bibinfo {year} {2014})}\BibitemShut {NoStop}%
\bibitem [{\citenamefont {Ro}\ and\ \citenamefont {Kim}(2017)}]{ro2017}%
  \BibitemOpen
  \bibfield  {author} {\bibinfo {author} {\bibfnamefont {S.}~\bibnamefont
  {Ro}}\ and\ \bibinfo {author} {\bibfnamefont {Y.~W.}\ \bibnamefont {Kim}},\
  }\href@noop {} {\bibfield  {journal} {\bibinfo  {journal} {Phys Rev E}\
  }\textbf {\bibinfo {volume} {96}},\ \bibinfo {pages} {012143} (\bibinfo
  {year} {2017})}\BibitemShut {NoStop}%
\bibitem [{\citenamefont {Lawley}(2019)}]{lawley2019gumbel}%
  \BibitemOpen
  \bibfield  {author} {\bibinfo {author} {\bibfnamefont {S.~D.}\ \bibnamefont
  {Lawley}},\ }\href@noop {} {\bibfield  {journal} {\bibinfo  {journal} {arXiv
  preprint arXiv:1910.12170}\ } (\bibinfo {year} {2019})}\BibitemShut {NoStop}%
\bibitem [{\citenamefont {Varadhan}(1967)}]{varadhan1967}%
  \BibitemOpen
  \bibfield  {author} {\bibinfo {author} {\bibfnamefont {S.~R.~S.}\
  \bibnamefont {Varadhan}},\ }\href@noop {} {\bibfield  {journal} {\bibinfo
  {journal} {Commun Pure Appl Math}\ }\textbf {\bibinfo {volume} {20}},\
  \bibinfo {pages} {659} (\bibinfo {year} {1967})}\BibitemShut {NoStop}%
\bibitem [{\citenamefont {Norris}(1997)}]{norris1997}%
  \BibitemOpen
  \bibfield  {author} {\bibinfo {author} {\bibfnamefont {J.~R.}\ \bibnamefont
  {Norris}},\ }\href@noop {} {\bibfield  {journal} {\bibinfo  {journal} {Acta
  Mathematica}\ }\textbf {\bibinfo {volume} {179}},\ \bibinfo {pages} {79}
  (\bibinfo {year} {1997})}\BibitemShut {NoStop}%
\bibitem [{\citenamefont {Grebenkov}(2006)}]{grebenkov2006}%
  \BibitemOpen
  \bibfield  {author} {\bibinfo {author} {\bibfnamefont {D.~S.}\ \bibnamefont
  {Grebenkov}},\ }\href@noop {} {\bibfield  {journal} {\bibinfo  {journal}
  {Focus on probability theory}\ ,\ \bibinfo {pages} {135}} (\bibinfo {year}
  {2006})}\BibitemShut {NoStop}%
\bibitem [{\citenamefont {Gardiner}(2009)}]{gardiner2009}%
  \BibitemOpen
  \bibfield  {author} {\bibinfo {author} {\bibfnamefont {C.}~\bibnamefont
  {Gardiner}},\ }\href {\doibase 10.1007/978-3-662-02377-8} {\emph {\bibinfo
  {title} {{Stochastic Methods: A Handbook for the Natural and Social
  Sciences}}}},\ \bibinfo {edition} {4th}\ ed.,\ \bibinfo {series} {Springer
  Series in Synergetics}, Vol.~\bibinfo {volume} {13}\ (\bibinfo  {publisher}
  {Springer Berlin Heidelberg},\ \bibinfo {year} {2009})\BibitemShut {NoStop}%
\bibitem [{\citenamefont {van Beijeren}(2003)}]{van2003}%
  \BibitemOpen
  \bibfield  {author} {\bibinfo {author} {\bibfnamefont {H.}~\bibnamefont {van
  Beijeren}},\ }\href@noop {} {\bibfield  {journal} {\bibinfo  {journal} {J
  Stat Phys}\ }\textbf {\bibinfo {volume} {110}},\ \bibinfo {pages} {1397}
  (\bibinfo {year} {2003})}\BibitemShut {NoStop}%
\bibitem [{\citenamefont {Lawley}\ and\ \citenamefont
  {Madrid}(2020)}]{lawley2019esp}%
  \BibitemOpen
  \bibfield  {author} {\bibinfo {author} {\bibfnamefont {S.~D.}\ \bibnamefont
  {Lawley}}\ and\ \bibinfo {author} {\bibfnamefont {J.~B.}\ \bibnamefont
  {Madrid}},\ }\href {\doibase 10.1007/s00332-019-09605-9} {\bibfield
  {journal} {\bibinfo  {journal} {Journal of Nonlinear Science}\ } (\bibinfo
  {year} {2020}),\ 10.1007/s00332-019-09605-9}\BibitemShut {NoStop}%
\bibitem [{\citenamefont {Blum}\ \emph {et~al.}(1989)\citenamefont {Blum},
  \citenamefont {Lawler}, \citenamefont {Reed},\ and\ \citenamefont
  {Shin}}]{blum1989}%
  \BibitemOpen
  \bibfield  {author} {\bibinfo {author} {\bibfnamefont {J.}~\bibnamefont
  {Blum}}, \bibinfo {author} {\bibfnamefont {G.}~\bibnamefont {Lawler}},
  \bibinfo {author} {\bibfnamefont {M.}~\bibnamefont {Reed}}, \ and\ \bibinfo
  {author} {\bibfnamefont {I.}~\bibnamefont {Shin}},\ }\href@noop {} {\bibfield
   {journal} {\bibinfo  {journal} {Biophys J}\ }\textbf {\bibinfo {volume}
  {56}},\ \bibinfo {pages} {995} (\bibinfo {year} {1989})}\BibitemShut
  {NoStop}%
\bibitem [{\citenamefont {Sykova}\ and\ \citenamefont
  {Nicholson}(2008)}]{sykova2008}%
  \BibitemOpen
  \bibfield  {author} {\bibinfo {author} {\bibfnamefont {E.}~\bibnamefont
  {Sykova}}\ and\ \bibinfo {author} {\bibfnamefont {C.}~\bibnamefont
  {Nicholson}},\ }\href@noop {} {\bibfield  {journal} {\bibinfo  {journal}
  {Physiological reviews}\ }\textbf {\bibinfo {volume} {88}},\ \bibinfo {pages}
  {1277} (\bibinfo {year} {2008})}\BibitemShut {NoStop}%
\bibitem [{\citenamefont {Nicholson}\ and\ \citenamefont
  {Hrabetova}(2017)}]{nicholson2017}%
  \BibitemOpen
  \bibfield  {author} {\bibinfo {author} {\bibfnamefont {C.}~\bibnamefont
  {Nicholson}}\ and\ \bibinfo {author} {\bibfnamefont {S.}~\bibnamefont
  {Hrabetova}},\ }\href@noop {} {\bibfield  {journal} {\bibinfo  {journal}
  {Biophys J}\ } (\bibinfo {year} {2017})}\BibitemShut {NoStop}%
\bibitem [{\citenamefont {Cherstvy}\ \emph {et~al.}(2013)\citenamefont
  {Cherstvy}, \citenamefont {Chechkin},\ and\ \citenamefont
  {Metzler}}]{cherstvy2013}%
  \BibitemOpen
  \bibfield  {author} {\bibinfo {author} {\bibfnamefont {A.~G.}\ \bibnamefont
  {Cherstvy}}, \bibinfo {author} {\bibfnamefont {A.~V.}\ \bibnamefont
  {Chechkin}}, \ and\ \bibinfo {author} {\bibfnamefont {R.}~\bibnamefont
  {Metzler}},\ }\href@noop {} {\bibfield  {journal} {\bibinfo  {journal} {New J
  Phys}\ }\textbf {\bibinfo {volume} {15}},\ \bibinfo {pages} {083039}
  (\bibinfo {year} {2013})}\BibitemShut {NoStop}%
\bibitem [{\citenamefont {Isaacson}\ \emph {et~al.}(2011)\citenamefont
  {Isaacson}, \citenamefont {McQueen},\ and\ \citenamefont
  {Peskin}}]{isaacson2011}%
  \BibitemOpen
  \bibfield  {author} {\bibinfo {author} {\bibfnamefont {S.}~\bibnamefont
  {Isaacson}}, \bibinfo {author} {\bibfnamefont {D.}~\bibnamefont {McQueen}}, \
  and\ \bibinfo {author} {\bibfnamefont {C.}~\bibnamefont {Peskin}},\
  }\href@noop {} {\bibfield  {journal} {\bibinfo  {journal} {Proc Natl Acad
  Sci}\ }\textbf {\bibinfo {volume} {108}},\ \bibinfo {pages} {3815} (\bibinfo
  {year} {2011})}\BibitemShut {NoStop}%
\bibitem [{\citenamefont {B{\'e}nichou}\ \emph {et~al.}(2010)\citenamefont
  {B{\'e}nichou}, \citenamefont {Chevalier}, \citenamefont {Klafter},
  \citenamefont {Meyer},\ and\ \citenamefont
  {Voituriez}}]{benichou2010geometry}%
  \BibitemOpen
  \bibfield  {author} {\bibinfo {author} {\bibfnamefont {O.}~\bibnamefont
  {B{\'e}nichou}}, \bibinfo {author} {\bibfnamefont {C.}~\bibnamefont
  {Chevalier}}, \bibinfo {author} {\bibfnamefont {J.}~\bibnamefont {Klafter}},
  \bibinfo {author} {\bibfnamefont {B.}~\bibnamefont {Meyer}}, \ and\ \bibinfo
  {author} {\bibfnamefont {R.}~\bibnamefont {Voituriez}},\ }\href@noop {}
  {\bibfield  {journal} {\bibinfo  {journal} {Nat Chem}\ }\textbf {\bibinfo
  {volume} {2}},\ \bibinfo {pages} {472} (\bibinfo {year} {2010})}\BibitemShut
  {NoStop}%
\bibitem [{\citenamefont {Woringer}\ \emph {et~al.}(2014)\citenamefont
  {Woringer}, \citenamefont {Darzacq},\ and\ \citenamefont
  {Izeddin}}]{woringer2014}%
  \BibitemOpen
  \bibfield  {author} {\bibinfo {author} {\bibfnamefont {M.}~\bibnamefont
  {Woringer}}, \bibinfo {author} {\bibfnamefont {X.}~\bibnamefont {Darzacq}}, \
  and\ \bibinfo {author} {\bibfnamefont {I.}~\bibnamefont {Izeddin}},\
  }\href@noop {} {\bibfield  {journal} {\bibinfo  {journal} {Curr Opin Chem
  Biol}\ }\textbf {\bibinfo {volume} {20}},\ \bibinfo {pages} {112} (\bibinfo
  {year} {2014})}\BibitemShut {NoStop}%
\bibitem [{\citenamefont {Dembo}\ and\ \citenamefont
  {Zeitouni}(1998)}]{dembo1998}%
  \BibitemOpen
  \bibfield  {author} {\bibinfo {author} {\bibfnamefont {A.}~\bibnamefont
  {Dembo}}\ and\ \bibinfo {author} {\bibfnamefont {.}~\bibnamefont
  {Zeitouni}},\ }\href@noop {} {\emph {\bibinfo {title} {Large deviations
  techniques and applications}}},\ \bibinfo {edition} {2nd}\ ed.\ (\bibinfo
  {publisher} {Springer},\ \bibinfo {year} {1998})\BibitemShut {NoStop}%
\bibitem [{\citenamefont {Pavliotis}(2014)}]{pavliotis2014}%
  \BibitemOpen
  \bibfield  {author} {\bibinfo {author} {\bibfnamefont {G.~A.}\ \bibnamefont
  {Pavliotis}},\ }\href@noop {} {\emph {\bibinfo {title} {Stochastic processes
  and applications: diffusion processes, the Fokker-Planck and Langevin
  equations}}},\ Vol.~\bibinfo {volume} {60}\ (\bibinfo  {publisher}
  {Springer},\ \bibinfo {year} {2014})\BibitemShut {NoStop}%
\bibitem [{\citenamefont {Lawley}\ and\ \citenamefont
  {Madrid}(2019)}]{lawley2019imp}%
  \BibitemOpen
  \bibfield  {author} {\bibinfo {author} {\bibfnamefont {S.~D.}\ \bibnamefont
  {Lawley}}\ and\ \bibinfo {author} {\bibfnamefont {J.~B.}\ \bibnamefont
  {Madrid}},\ }\href@noop {} {\bibfield  {journal} {\bibinfo  {journal} {J Chem
  Phys}\ }\textbf {\bibinfo {volume} {150}},\ \bibinfo {pages} {214113}
  (\bibinfo {year} {2019})}\BibitemShut {NoStop}%
\end{thebibliography}%

\end{document}